\documentclass{amsart}

\usepackage[]{fontenc}
\usepackage[latin1]{inputenc}
\usepackage{geometry}
\usepackage{fancyhdr}
\usepackage{setspace}
\usepackage[all]{xy}
\usepackage{amssymb}

\usepackage{amsfonts}
\usepackage{eufrak}
\usepackage{amsmath}
\usepackage{amsthm}
\usepackage{a4wide}
\usepackage{textcomp}
\usepackage{epsfig}

\numberwithin{equation}{section}

\newtheorem{thm}{Theorem}[section]
\newtheorem{prop}[thm]{Proposition}
\newtheorem{lemma}[thm]{Lemma}
\newtheorem{rmk}[thm]{Remark}
\newtheorem{ex}[thm]{Example}

\begin{document}

\title{Bi--invariant metrics on contactomorphism groups}

\author{Sheila Sandon}
\address{Institut de Recherche Math\'{e}matique Avanc\'{e}e\\
Universit\'{e} de Strasbourg - France}
\email{sandon@math.unistra.fr}

\begin{abstract}
\noindent
Contact manifolds are odd--dimensional smooth manifolds endowed with a maximally non--integrable field of hyperplanes. They are intimately related to symplectic manifolds, i.e. even--dimensional smooth manifolds endowed with a closed non--degenerate 2--form. Although in symplectic topology a famous bi--invariant metric, the Hofer metric, has been studied since more than 20 years ago, it is only recently that some somehow analogous bi--invariant metrics have been discovered on the group of diffeomorphisms that preserve a contact structure. In this expository article I will review some constructions of bi--invariant metrics on the contactomorphism group, and how these metrics  are related to some other global rigidity phenomena in contact topology which have been discovered in the last few years, in particular the notion of orderability (due to Eliashberg and Polterovich) and an analogue in contact topology (due to Eliashberg, Kim and Poltorovich) of Gromov's symplectic non-squeezing theorem.
\end{abstract}

\maketitle


\section{Introduction}

Given a group $G$, a function $d: G \times G \rightarrow [0,\infty)$ is said to be a \emph{bi--invariant metric} if it satisfies the following properties:
\begin{itemize}
\item $d(g_1,g_2) \geq 0$, with equality if $g_1 = g_2$ (positivity);
\item if $d(g_1,g_2) = 0$ then $g_1 = g_2$ (non--degeneracy);
\item $d(g_1,g_2) = d(g_2,g_1)$ (symmetry);
\item $d(g_1,g_3) \leq d(g_1,g_2) + d(g_2,g_3)$ (triangle inequality); 
\item $d(fg_1,fg_2) = d(g_1,g_2) = d(g_1f,g_2f)$ (bi--invariance).
\end{itemize}
We say that $d$ is a \emph{bi--invariant pseudometric} if it satisfies all the above properties except possibly for non--degeneracy. Any bi--invariant (pseudo)metric $d$ on $G$ induces a \emph{conjugation--invariant (pseudo)norm} $\lVert \, \cdot \, \rVert: G \rightarrow [0,\infty)$ by posing $\lVert g \rVert = d(g,\text{id})$. Conversely, any conjugation--invariant (pseudo)norm $\lVert \, \cdot \, \rVert$ on $G$ induces a bi--invariant (pseudo)metric $d$ by posing $d(g_1,g_2) = \lVert g_1^{-1}g_2 \rVert$. In the following we will use these two notions interchangeably. Two conjugation--invariant (pseudo)\linebreak norms on the same group are said to be equivalent if their ratio is bounded and bounded away from zero.

In the present article we will be interested in bi--invariant metrics on the identity component of the group of contactomorphisms of a contact manifold $(M,\xi)$, i.e. the group of diffeomorphisms of $M$ that preserve the contact distribution $\xi$. Note that on the diffeomorphism group of a smooth manifold there exist several bi--invariant metrics but, as discussed by Burago, Ivanov and Polterovich \cite{BIP}, they all seem to be always equivalent to the trivial metric. On the other hand, in 1990 Hofer discovered \cite{Hofer} a remarkable bi--invariant metric on the group of Hamiltonian diffeomorphisms of any symplectic manifold. Since its discovery the Hofer metric has been the object of much research, especially in relation to other global rigidity phenomena in symplectic topology such as for example Gromov's famous non--squeezing theorem \cite{Gromov}. Another bi--invariant metric on the Hamiltonian group, related but not equivalent to the Hofer metric, was later defined by Viterbo \cite{Viterbo} for the standard symplectic Euclidean space $\mathbb{R}^{2n}$ and extended by  Schwarz \cite{Schwarz} and Oh \cite{Oh} to more general symplectic manifolds.

Although, as we will review in Section \ref{section: Contact and symplectic manifolds}, contact manifolds are deeply related to symplectic manifolds, until recently no bi--invariant metrics were known on the group of contactomorphisms. The first such bi--invariant metric was discovered in 2010 \cite{mio 2} for the manifold $\mathbb{R}^{2n} \times S^1$ with its standard contact structure, as a generalization to contact topology of the Viterbo metric. Other constructions of bi--invariant metrics were later given for more general contact manifolds by Zapolsky \cite{Zapolsky}, Fraser, Polterovich and Rosen \cite{FPR} and Colin and the present author \cite{CS}. For many aspects these bi--invariant metrics on the contactomorphism group can be seen as analogues of the Hofer and Viterbo metrics on the Hamiltonian group of a symplectic manifold. On the other hand, as we will discuss, they also present some specificities that make them still quite mysterious and not so well understood.

In this expository article we will review the known constructions of bi--invariant metrics on the contactomorphism group and discuss their relation to other surprising contact rigidity phenomena such as the notion of orderability, that was introduced in 2000 by Eliashberg and Polterovich \cite{EP00}, and a contact analogue, discovered in 2006 by Eliashberg, Kim and Polterovich \cite{EKP}, of Gromov's symplectic non--squeezing theorem. In the next section we will give some preliminaries on contact manifolds, symplectic manifolds and their interactions. In Section \ref{section: symplectic rigidity} we will discuss the Hofer and Viterbo bi--invariant metrics on the Hamiltonian group of a symplectic manifold, and their relation to other global rigidity results in symplectic topology. In Section \ref{section: Hofer-like lengths} we will show why the formula that defines the Hofer metric does not give rise in the contact case to a bi--invariant metric, and in Section \ref{section: Rigidity and flexibility in contact topology} we will prove, following Polterovich \cite{Polterovich - private}, that on the contactomorphism group it is actually impossible to have a conjugation--invariant norm that takes values arbitrarily close to zero (as it is the case for the Hofer and Viterbo metrics on the Hamiltonian group). As we will see, the impossibility to have such a bi--invariant metric on the contactomorphism group is due to the fact that in the Euclidean space $\mathbb{R}^{2n+1}$ with its standard contact structure it is possible to squeeze any given domain into an arbitrarily small one. On the other hand, in Section \ref{section: Rigidity and flexibility in contact topology} we will also discuss a contact non--squeezing phenomenon on the contact manifold $\mathbb{R}^{2n} \times S^1$, a result that was discovered by Eliashberg, Kim and Polterovich \cite{EKP} and motivated the search for \emph{discrete} bi--invariant metrics on the contactomorphism group. In Section \ref{section: spectral metric} we will review, following \cite{mio 2}, how the construction of the Viterbo metric on the Hamiltonian group of $\mathbb{R}^{2n}$ can be naturally generalized to the contact case, giving rise to an \emph{integer--valued} bi--invariant metric on the contactomorphism group of $\mathbb{R}^{2n} \times S^1$. As we will see the key idea is that, in the contact case, \emph{translated points} of contactomorphims play the role that is played in the symplectic case by fixed points of Hamiltonian symplectomorphims. Following \cite{CS}, in Section \ref{section: discriminant metric} we will then discuss how this idea can be used to define an integer--valued bi--invariant metric (the \emph{discriminant metric}) on the universal cover of the contactomorphism group of any contact manifold. In Section \ref{section: orderability and the oscillation metric} we will review the notion of \emph{orderability} in contact topology, which has been introduced in 2000 by Eliashberg and Polterovich \cite{EP00}, and (still following \cite{CS}) discuss how to combine this notion with the discriminant metric in order to define what we call the \emph{oscillation metric}. We will conclude by indicating in Section \ref{section: Discussion} some open questions and future directions of research.

\subsection*{Acknowledgments}  This article has been written in occasion of the 5-th IST-IME meeting in S\~{a}o Paulo in honor of Prof. Orlando Lopes. I would like to thank the organizers for inviting me to participate to such a pleasant and interesting event. I also would like to thank the referee for several corrections and useful comments. The research on which the article is based has been carried out from 2005 to 2009 at the Instituto Superior T\'{e}cnico in Lisbon, in 2010 and 2011 at the Laboratoire de Math\'{e}matiques Jean Leray in Nantes, in the spring semester of 2012 at the Institute for Advanced Study in Princeton, from September 2012 to August 2014 at the UMI--CNRS of the Centre de Recherches Math\'{e}matiques in Montr\'{e}al and since September 2014 at the Institut de Recherche Math\'{e}matique Avanc\'{e}e in Strasbourg. Since January 2014 I am supported by the ANR grant COSPIN. I would like to thank all members of this project for very interesting discussions about bi--invariant metrics and related topics.


\section{Contact and symplectic manifolds}\label{section: Contact and symplectic manifolds}

In this section we review some basic notions in symplectic and contact topology. More material on these subjects can be found for example in the books of Geiges \cite{Geiges}, McDuff and Salamon \cite{MS}, Hofer and Zehnder \cite{HZ} and Cieliebak and Eliashberg \cite{CE}.

A \textbf{contact manifold} is a smooth odd--dimensional manifold $M^{2n+1}$ endowed with a field of hyperplanes $\xi$ (called the \emph{contact structure}) which is co--oriented and maximally non--integrable. These two conditions mean that  $\xi$ can be written as the kernel of a 1--form $\alpha$ (which is called a \emph{contact form}) such that $\alpha \wedge (d\alpha)^n$ is a volume form. Note that the fact that $\alpha \wedge (d\alpha)^n$ is a volume form can equivalently be expressed by saying that $d\alpha \rvert_{\xi}$ is non--degenerate, which is the opposite of the Frobenius integrability condition (recall that, by the Frobenius theorem, if a distribution $\xi$ can be written as the kernel of a 1--form $\alpha$ then $\xi$ is integrable if and only if $d\alpha \rvert_{\xi}$ vanishes identically). 

The standard example of a contact manifold is the Euclidean space $\mathbb{R}^{2n+1}$ with the contact structure given by 
$$
\xi_0 = \text{ker} \big(dz + \sum_{i=1}^n x_idy_i - y_idx_i\big)\,.
$$
Note that the distribution $\xi_0$ is invariant by translations in the $z$--direction. It is easy to visualize it for example in dimension $3$: at the origin it is the horizontal plane $\text{ker}(dz)$, while when we move along a line through the origin the distribution twists, and becomes vertical at $\pm \infty$.  By the Darboux theorem, the contact manifold $(\mathbb{R}^{2n+1},\xi_0)$ is the local model for any other contact manifold of the same dimension, in the sense that any point of any $(2n+1)$--dimensional contact manifold has a neighborhood that is diffeomorphic to a neighborhood of the origin of $(\mathbb{R}^{2n+1},\xi_0)$. Thus, \emph{locally} all contact manifolds of the same dimension are equivalent. On the other hand an important topic of research since the work of Bennequin \cite{Bennequin} is to find examples of contact manifolds that are \emph{globally} non--equivalent (even in cases when the underlying smooth manifolds are diffeomorphic, and the contact structures are in the same homotopy class as hyperplane fields) and to construct global invariants capable to distinguish them. In the present article we will not deal with the problem of classification of contact structures, but we will discuss instead some flexibility and rigidity phenomena\footnote{Flexibility and rigidity in contact topology appear in several a priori quite different flavors. See for example \cite{Eliashberg - Flexibility} for a survey of recent discoveries on the flexibility side, in particular related to the notion of \emph{loose Legendrian submanifolds} \cite{Murphy} and \emph{overtwisted} contact manifolds \cite{BEM}. It is still an open question to understand whether these notions of flexibility interact in some interesting way with the global rigidity phenomena that are discussed in this paper. For example, a concrete question in this direction would be to understand whether or not any bi--invariant metric on the contactomorphism group of an overtwisted contact manifold is necessarily equivalent to the trivial metric.} for diffeomorphisms that preserve the contact structure on some fixed contact manifold. As we will see, the global rigidity phenomena that we will discuss are reminiscent of analogue global rigidity phenomena in symplectic topology. 

A \textbf{symplectic manifold} is an even--dimensional smooth manifold $W^{2n}$ endowed with a non--degenerate closed 2--form $\omega$ (which is called a \emph{symplectic form}). The standard example of a symplectic manifold is the Euclidean space $\mathbb{R}^{2n}$ with the symplectic form
$$
\omega_0 = \sum_{i=1}^n dx_i \wedge dy_i\,.
$$
By the Darboux theorem, all $(2n)$--dimensional symplectic manifolds are locally equivalent to a neighborhood of the origin in $(\mathbb{R}^{2n},\omega_0)$. On the other hand, starting from the 80's some surprising global rigidity phenomena has been discovered in symplectic topology. We will review three of them in Section \ref{section: symplectic rigidity}.

Contact manifolds are intimately related to symplectic manifolds. Suppose in particular that $(W,\omega)$ is an \emph{exact} symplectic manifold, i.e. $\omega = d\lambda$ for some 1--form $\lambda$. Let $Y_{\lambda}$ be the \emph{Liouville vector field} associated to $\lambda$, i.e. the vector field defined by the relation $\iota_{Y_{\lambda}}\omega = \lambda$. Then any hypersurface $M$ of $W$ which is transverse to $Y_{\lambda}$ is a contact manifold, with contact form $\alpha := \lambda \rvert_M$. 

\begin{ex}
The symplectic Euclidean space $(\mathbb{R}^{2n},\omega_0)$ is exact, since $\omega_0 = d\lambda_0$ with $\lambda_0 = \frac{1}{2} \sum_{i=0}^n x_i \frac{\partial}{\partial x_i} + y_i \frac{\partial}{\partial y_i}$. The Liouville vector field $Y_{\lambda_0}$ spans the radial direction and so it is transverse to the unit sphere $S^{2n-1}$. The kernel of the 1--form   $\alpha_0 := \lambda_0\rvert_{S^{2n-1}}$ is said to be  the standard contact structure on $S^{2n-1}$. 
\end{ex}

Another way to obtain contact manifolds from symplectic manifolds is as follows. We start again with an exact symplectic manifold $(W,\omega= d\lambda)$, but instead of taking a hypersurface we now add a dimension and consider the manifold $P = W \times \mathbb{R}$ (or $W \times S^1$). Then the distributon $\xi =\text{ker}(dz - \lambda)$, where $z$ is the coordinate in the $\mathbb{R}$-- (or $S^1$--)direction, is a contact structure on $P$. The contact manifold $(P,\xi)$ is called the \emph{prequantization}\footnote{This construction can be generalized to all symplectic manifolds $(W,\omega)$ with $[\omega] \in H^2(W;\mathbb{Z}_2)$, to define a contact structure on the principal $S^1$--bundle over $W$ with Euler class $[\omega]$. See Boothby and Wang \cite{BW} or \cite[Section 7.2]{Geiges}.}, or \emph{contactification}, of $(W,\omega)$. 

\begin{ex}
The contact Euclidean space $(\mathbb{R}^{2n+1},\xi_0)$ is the prequantization of the standard symplectic Euclidean space $(\mathbb{R}^{2n},\omega_0)$. 
\end{ex}

\begin{ex}
The cotangent bundle $T^{\ast}B$ of a smooth manifold $B$ has a canonical exact symplectic form $\omega_{\text{can}} = - d \lambda_{\text{can}}$, where the Liouville form $\lambda_{\text{can}}$ is defined by $\lambda_{\text{can}}(X) = \sigma \big(\pi_{\ast}(X)\big)$ for a vector $X$ in $T_{(q,\sigma)} (T^{\ast}B)$. The prequantization of $T^{\ast}B$ is the 1--jet bundle $J^1B = T^{\ast}B \times \mathbb{R}$ (with the induced contact form $dz - \lambda_{\text{can}})$.
\end{ex}

A third link between symplectic and contact topology is given by the fact that, for a contact manifold $(M,\xi)$ with contact form $\alpha$, the 2--form $d\alpha$ is a symplectic form on the contact distribution $\xi$. In other words, $\xi$ can be seen as a symplectic vector bundle over $M$. On the other hand if we look at the 2--form $d\alpha$ on the whole tangent bundle $TM$ then we see that, for reasons of linear algebra, $d\alpha$ must have a 1--dimensional kernel. This 1--dimensional kernel is called the \emph{Reeb direction} associated to the contact form $\alpha$. We define the \textbf{Reeb vector field} $R_{\alpha}$ to be the vector field that spans the Reeb direction and is normalized by the condition $\alpha(R_{\alpha}) = 1$. The flow of $R_{\alpha}$ preserves the contact structure $\xi$ (in fact, it even preserves the contact form $\alpha$). 

\begin{ex}
The Reeb flow associated to the standard contact forms in $\mathbb{R}^{2n+1}$ and $J^1B$ is given by translation in the z--direction. On the sphere $S^{2n-1}$, the Reeb flow associated to the standard contact form is given by rotation along the Hopf fibers.
\end{ex}

We emphasize that the Reeb vector field $R_{\alpha}$  on a contact manifold $\big(M,\xi = \text{ker}(\alpha)\big)$ depends on the choice of a contact form $\alpha$ for $\xi$ (note that such choice is not unique, indeed if $\alpha$ is a contact form for $\xi$ then so is $e^f\alpha$ for any function $f$ on $M$). Once we choose a contact form $\alpha$ for $\xi$, we can decompose the tangent bundle of $M$ as the direct sum
\begin{equation}\label{equation: decomposition of TM}
TM = \xi \;\oplus \left< \, R_{\alpha} \, \right>\,.
\end{equation}
A diffeomorphism $\phi$ of $M$ is said to be a \emph{contactomorphism} if it preserves the contact structure $\xi$, i.e. $\phi_{\ast} \xi = \xi$ or in other words $\phi^{\ast} \alpha = e^g \alpha$ for some function $g$ on $M$. Note that, since a contactomorphism is not required to preserve the contact form, in general it does not preserve the associated Reeb vector field. Thus, in terms of the decomposition (\ref{equation: decomposition of TM}), a contactomorphism preserves the first direct summand but not necessarily the second. As we have seen, the first direct summand is a symplectic object (the 2--form $d\alpha$ is a symplectic form on $\xi$) while from the point of view of symplectic topology the second summand should be seen as a degenerate direction (since the 2--form $d\alpha$ vanishes on it). In the topic that we will discuss in this article, bi--invariant metrics on the contactomorphism group, we will see the influence of both the symplectic and the degenerate parts of the tangent bundle of a contact manifold. We will see that the symplectic side tends to give rise to some global rigidity, which at a first sight seems to be completely destroyed by the flexibility given by the degenerate Reeb direction. However we will explain that, at a closer look, some rigidity does survive, making it possible in particular to define non--trivial bi--invariant metrics on the group of contactomorphisms. Such bi--invariant metrics are reminiscent of the analogous metrics on the Hamiltonian group of a symplectic manifold. However we will also discuss two aspects that distinguish them, as well as other related contact rigidity results, with respect to their symplectic analogues. The first aspect is that in the contact case the rigidity phenomena that we will discuss have a \emph{discrete character} (in particular, bi--invariant metrics are \emph{integer--valued}, while their analogues in the symplectic case take values that are arbitrarily close to zero). The second specificity is that in the contact case the rigidity phenomena that we will discuss are also sensitive to the topology of the underlying manifold.

In the next section we will start the discussion by describing three famous global rigidity results in symplectic topology: Gromov's non--squeezing theorem \cite{Gromov}, the Arnold conjecture on fixed points  of Hamiltonian diffeomorphisms and the existence of bi--invariant metrics on the Hamiltonian group \cite{Hofer, Viterbo}. In the rest of the article we will then discuss what these phenomena become in the contact case, concentrating in particular on bi--invariant metrics.


\section{Symplectic rigidity and bi--invariant metrics on the Hamiltonian group}\label{section: symplectic rigidity}

Let $(W,\omega)$ be a symplectic $2n$--dimensional manifold. Since $\omega$ is non--degenerate we have that $\omega^n$ is a volume form. Thus, all transformation that preserve the symplectic form in particular also preserve the volume. In this section we will discuss three classical symplectic rigidity results that, on the other hand, distinguish symplectic transformations from volume--preserving ones.

In the standard symplectic Euclidean space $(\mathbb{R}^{2n},\omega_0)$ consider the domains
$$
B^{2n}(R) = \{\,\pi \sum_{i=1}^n x_i^2 + y_i^2 < R\,\}
$$
and $C^{2n}(R) = B^2(R) \times \mathbb{R}^{2n-2}$. \textbf{Gromov's non--squeezing theorem} \cite{Gromov} says that if $R_2 < R_1$ there is no symplectic embedding of the ball $B^{2n}(R_1)$ into the cylinder $C^{2n}(R_2)$. Note that, since $C^{2n}(R_2)$ has infinite volume, it is certainly possible to find a volume--preserving embedding of $B^{2n}(R_1)$ into $C^{2n}(R_2)$. Gromov's theorem thus tells us that such volume preserving embedding cannot be approximated by symplectic ones. This result showed for the first time that being a symplectic transformation is a much stricter and fundamentally different condition than just preserving volume, and is therefore often considered as the beginning of modern symplectic topology. 

A second result that is similar in spirit to Gromov's non--squeezing theorem (and in fact deeply related to it) is the existence of bi--invariant metrics on the Hamiltonian group of a symplectic manifold $(W,\omega)$. The Hamiltonian group is a subgroup of the group $\text{Symp}(W,\omega)$ of \textit{symplectomorphisms}, i.e. diffeomorphisms of $W$ that preserve $\omega$.  It is defined as follows. Note first that every smooth function $H: W \rightarrow \mathbb{R}$ induces a 1--parameter subgroup of $\text{Symp}(W,\omega)$, indeed the flow of the vector field $X_H$ determined by the relation
$$\iota_{X_H}\omega=-dH$$
preserves $\omega$ at all times. The vector field $X_H$ is called the \textit{Hamiltonian vector field} associated to the (autonomous) \emph{Hamiltonian function} $H$. More generally we can also consider the \textit{Hamiltonian flow} of a time-dependent function $H_t: W \rightarrow \mathbb{R}$, i.e. the flow of the time-dependent vector field $X_{H_t}$ defined by $\iota_{X_{H_t}}\omega = dH_t$. Again, the flow $\varphi_t$ of $X_{H_t}$ preserves $\omega$. An isotopy $\varphi_t$ of $(W,\omega)$ is called a \textit{Hamiltonian isotopy} if it is the Hamiltonian flow of a (time-dependent) function $H_t$. A \textbf{Hamiltonian diffeomorphism} of $(W,\omega)$ is a symplectomorphism  that can be written as the time-1 map of a Hamiltonian isotopy. We denote by $\text{Ham}(W,\omega)$ the group of Hamiltonian symplectomorphisms, and by $\widetilde{\text{Ham}}(W,\omega)$ its universal cover, i.e. the space of homotopy classes (with fixed endpoints) of Hamiltonian isotopies starting at the identity. 

Note that the Hamiltonian function of a given Hamiltonian isotopy is only well--defined up to the addition of a constant. To remove this ambiguity we normalize Hamiltonian functions, either by requiring $\int_W H_t \, \omega^n = 0$ for all $t$ (in the case when $W$ is compact) or (when we study compactly supported Hamiltonian symplectomorphisms of an open symplectic manifold) by requiring the Hamiltonian function to be compactly supported. 

In 1990 Hofer \cite{Hofer} discovered a \textbf{bi--invariant metric} on (the universal cover of) the group of Hamiltonian diffeomorphisms of a compact symplectic manifold (or the group of compactly supported Hamiltonian diffeomorphisms, in the case when the symplectic manifold is not compact). The Hofer metric is defined as follows. We first define the length of the Hamiltonian isotopy $\{\varphi_t\}_{t\in[0,1]}$ generated by a Hamiltonian function $H_t$ as
$$
l({\varphi_t}) = \int_0^1 \max_{x\in W} H_t(x) - \min_{x\in W} H_t(x)\,dt\,.
$$
We then define the norm $\lVert \varphi \rVert$ of a Hamiltonian symplectomorphism $\varphi$ to be the infimum of the lengths $l({\varphi_t})$ of all Hamiltonian isotopies $\varphi_t$ from the identity to  $\varphi$. Note that this definition mimics the usual definition of the distance between two points in a Riemannian manifold: we first define the length of a path between two points as the integral with respect to time of the norm of the tangent vector, and then define the distance between two points as the infimum of the length of all paths connecting them. In our case we work on the infinite--dimensional space $\text{Ham}(W,\omega)$ of Hamiltonian diffeomorphisms. A path between two points in $\text{Ham}(W,\omega)$ is given by a Hamiltonian isotopy between them\footnote{In fact, we are applying here a deep result due to Banyaga \cite{Banyaga}: any path $\varphi_t$ in $\text{Symp}(W,\omega)$ such that $\varphi_t$ is in $\text{Ham}(W,\omega)$ for all $t$ is a Hamiltonian isotopy.} and, as discussed above, the tangent space to $\text{Ham}(W,\omega)$ can be identified to the space of (time--dependent) functions on $W$ by the correspondence that associates to a function $H_t$ on $W$ the generated Hamiltonian isotopy. The Hofer metric can then be seen as the metric on $\text{Ham}(W,\omega)$ that is induced by the $L_{\infty}$--norm\footnote{Regarding the choice of the norm on $\mathcal{C}^{\infty}(W)$, Eliashberg and Polterovich \cite{EP - Bi-invariant metrics} showed that for any finite $p$ the pseudometric on $\text{Ham}(W,\omega)$ induced by the $L_p$--norm on $\mathcal{C}^{\infty}$(W) is always degenerate, and if $W$ is compact it vanishes identically. Moreover Buhovsky and Ostrover \cite{BO} proved that, if $W$ is compact, the bi--invariant pseudometric on $\text{Ham} (W,\omega)$ induced by any invariant pseudonorm on $\mathcal{C}^{\infty}(W)$ which is continuous with respect to the $\mathcal{C}^{\infty}$--topology either vanishes identically or is equivalent to the Hofer metric.} $\lVert H \rVert_{\infty} = \max_{x\in W}H - \min_{x \in W} H$ on $\mathcal{C}^{\infty}(W)$. Similarly we define the norm $\lVert [\{(\varphi_t\}] \rVert$ of an element of $\widetilde{\text{Ham}}(W,\omega)$ to be the infimum of the length of all isotopies representing the element.

It is quite easy to prove (see for example \cite[Section 12.3]{MS}) that the above definition gives rise to a bi--invariant pseudometric on $\text{Ham} (W,\omega)$ and $\widetilde{\text{Ham}}(W,\omega)$. On the other hand, non--degeneracy of the Hofer metric is a very deep fact which has been proved by Hofer \cite{Hofer} for $W = \mathbb{R}^{2n}$ and by Lalonde and McDuff \cite{Lalonde-McDuff} in general. Note that the fact that the Hamiltonian group is infinite--dimensional can be seen as a manifestation of the local flexibility of symplectic manifolds. This is in contrast to what happens for instance in Riemannian geometry, where the group of isometries of a compact manifold is always finite--dimensional. On the other hand, the existence of some non--trivial geometric structures on the Hamiltonian group (in particular bi--invariant metrics) can be interpreted as a manifestation of global rigidity. In fact, as explained in the work of Lalonde and McDuff, non--degeneracy of the Hofer metric is deeply related to the symplectic non--squeezing theorem. Note that the Hofer norm is never equivalent to the trivial norm, because it takes values arbitrarily close to zero. Moreover the Hofer norm is conjectured\footnote{For compact symplectic manifolds $(W,\omega)$ with $\pi_2(W) = 0$ unboundedness of the Hofer metric has been proved by Ostrover \cite{Ostrover - A comparison}. See for example Polterovich \cite{Polterovich - Geometry on the group of Hamiltonian diffeomorphisms} for more references of other known results.} to be unbounded for all symplectic manifolds (even for compact symplectic manifolds, in contrast to what happens for instance for the \emph{size--of--support norm} of volume preserving diffeomorphisms \cite{BIP}). We refer to \cite{Polterovich - book} for a monograph on the Hofer norm.

A second bi--invariant metric was discovered in 1992 by Viterbo \cite{Viterbo} on the Hamiltonian group of Euclidean space $\mathbb{R}^{2n}$ (with its standard symplectic form). The Viterbo metric metric is also related to the non--squeezing theorem, and was used in \cite{Viterbo} to give an alternative proof of non--degeneracy of the Hofer metric. The construction of Viterbo was then extended to more general symplectic manifolds by Schwarz \cite{Schwarz}  and Oh \cite{Oh}. The idea of the construction given by Viterbo is as follows. For a Hamiltonian symplectomorphism $\varphi$ of $\mathbb{R}^{2n}$ we define the \textit{symplectic action} of a fixed point $q$ by
$$
\mathcal{A}_{\varphi}(q) = \int_0^1 \big(\lambda(X_t) + H_t\big) \big(\varphi_t(q)\big) \, dt
$$
where $\varphi_t$ is a Hamiltonian isotopy joining $\varphi$ to the identity, $X_t$ is the vector field generating it and $H_t$ the corresponding Hamiltonian function. It can be shown that this value does not depend on the choice of a Hamiltonian isotopy $\varphi_t$ from the identity to $\varphi$. Note that $\mathcal{A}_{\varphi}(q)$ is the value at the loop $\varphi_t(q)$ of the \textbf{symplectic action functional } $\mathcal{A}_H$ on the space of loops of $\mathbb{R}^{2n}$, i.e. the functional which is defined by
$$
\mathcal{A}_H(\gamma) = \int_0^1 \Big( \lambda_0\big(\dot{\gamma}(t) + H_t\big(\gamma(t)\big)\big)\Big)\,dt
$$
for a loop $\gamma: [0,1] \rightarrow \mathbb{R}^{2n}$. The symplectic action functional plays a crucial role in symplectic topology, as important as the one played for example by the length or energy functional in Riemannian geometry. The definition of the symplectic action functional, which is given above only in the case of an exact symplectic manifold, can in fact be generalized to all symplectic manifolds. The study of the symplectic action functional led to the discovery of several important tools of symplectic topology, such as for instance Floer homology and generating functions. For a Hamiltonian symplectomorphism $\varphi$ of $\mathbb{R}^{2n}$ we define its \emph{action spectrum} $\Lambda(\varphi)$ to be the set of values $\mathcal{A}_{\varphi}(q)$ for all fixed points $q$ of $\varphi$. As we will review in Section \ref{section: spectral metric}, in \cite{Viterbo} Viterbo used generating functions to define \textit{spectral invariants} $c^+(\varphi)$ and $c^-(\varphi)$ for a compactly supported Hamiltonian symplectomorphism $\varphi$ of $\mathbb{R}^{2n}$, i.e. real numbers belonging to the action spectrum of $\varphi$ and satisfying certain properties that will also be reviewed in Section \ref{section: spectral metric}. Using these spectral invariants Viterbo then obtained several important applications. In particular, for a domain $\mathcal{U}$ of $\mathbb{R}^{2n}$ he defined a \emph{symplectic capacity} $c(\mathcal{U})$ by taking the supremum of $c^+(\varphi)$ for all Hamiltonian symplectomorphisms $\varphi$ supported in $\mathcal{U}$, and used it to obtain a new proof of the symplectic non--squeezing theorem. As we will see in more details in Section \ref{section: spectral metric}, the Viterbo bi--invariant metric on the group of compactly supported Hamiltonian symplectomorphism is also defined in terms of the spectral invariants $c^\pm$.

We conclude this section by mentioning a third famous global rigidity result in symplectic topology, the \textbf{Arnold conjecture} on fixed points of Hamiltonian symplectomorphisms. Let $\varphi$ be a Hamiltonian symplectomorphism of a symplectic manifold $(W,\omega)$. The graph
$$
\text{gr}(\varphi) := \{\,(q,\varphi(q)) \; | \; q \in W\,\}
$$
of $\varphi$ is a \emph{Lagrangian submanifold} of the symplectic product $\overline{W} \times W := \big(W \times W, (-\omega) \oplus \omega\big)$, i.e. a submanifold on which the symplectic form vanishes and has maximal dimension with respect to this property (thus, since $\omega$ is non--degenerate, has dimension equal to half the dimension of the ambient symplectic manifold). By the Weinstein theorem, every Lagrangian submanifold has a tubular neighborhood that is symplectomorphic to a neighborhood of the 0--section on its cotangent bundle. In particular we know thus that the diagonal $\Delta := \text{gr}(\text{id})$ of $\overline{W} \times W$ has a neighborhood which is symplectomorphic to a neighborhood of the 0--section in $T^{\ast}\Delta$. If $\varphi$ is a $\mathcal{C}^1$--small Hamiltonian symplectomorphism then its graph $\text{gr}(\varphi)$ is an exact Lagrangian section in this neighborhood and so it can be written as the graph of the differential of a function $f: \Delta \rightarrow \mathbb{R}$. Critical points of $f$ correspond to intersections of $\Delta$ with $\text{gr}(\varphi)$, hence to fixed points of $\varphi$. The conclusion is thus that if $\varphi$ is a $\mathcal{C}^1$--small Hamiltonian symplectomorphism of a compact symplectic manifold then $\varphi$ always has fixed points, at least as many as the minimal number of critical points of a function on it (note that this number is in general larger than the one predicted by the Lefchetz fixed point theorem). The Arnold conjecture states that the above existence result for fixed points should be true for all Hamiltonian symplectomorphisms of a compact symplectic manifold. By applying a new approach of infinite--dimensional Morse theory to the symplectic action functional (developing what is now known as Floer theory) Floer \cite{Floer} proved a non--degenerate version of the Arnold conjecture in the case of \emph{monotone} symplectic manifolds. Floer's method was then generalized by Fukaya and Ono \cite{FO}, Liu and Tian \cite{LT} and Hofer and Salamon \cite{HS} to prove the non--degenerate version of the Arnold conjecture for all symplectic manifolds. For certain symplectic manifolds it is also possible to prove the Arnold conjecture by doing classical (finite--dimensional) Morse theory on \emph{generating functions} (which are finite--dimensional reductions of the symplectic action functional) of Hamiltonian symplectomorphisms (see for instance Chaperon \cite{Ch1}  and Th\'{e}ret \cite{Th2}).
 
We have seen in this section three global rigidity results in symplectic topology. In the rest of the article we will first discuss how in the contact case the flexibility given by the Reeb flow seems to destroy all of them, and then see how still some global rigidity survives. We will start by discussing in the next section what happens in the contact case if we take the same formula that we used to define the Hofer metric.


\section{Hofer--like lengths for contact isotopies}\label{section: Hofer-like lengths}

Let $(M,\xi)$ be a contact manifold. As in the symplectic case, we have that any (possibly time--dependent) function on $M$ induces a flow of contactomorphisms. Indeed, let $H_t: M \rightarrow \mathbb{R}$ be a (time--dependent) function. Then $H_t$ uniquely defines a vector field $X_t$ by the formulas
\begin{equation}\label{equation: contact Hamiltonians}
\left\{
\begin{array}{l}
\iota_{X_t} d\alpha = dH_t(R_{\alpha}) \alpha - dH_t\\
\alpha(X_t) = H_t\,.
\end{array} \right.
\end{equation}
The flow $\{\phi_t\}$ of $X_t$ is a \textbf{contact isotopy}, i.e. every $\phi_t$ is a contactomorphism. Note that, in contrast to the symplectic case, all contact isotopies can be written as the flow induced by a time--dependent function on $M$ (see for example \cite[Section 2.3]{Geiges}). Note also that the contact Hamiltonian function of a contact isotopy is uniquely defined, since if we add a constant to the function the generated isotopy also changes, because of the second equation in (\ref{equation: contact Hamiltonians}). We have thus a 1--1 correspondence between contact isotopies and Hamiltonian functions (which however depends on the choice of a contact form $\alpha$ for $\xi$).

While in the symplectic case a constant Hamiltonian function generates the Hamiltonian isotopy that is constantly the identity, we see from the equations (\ref{equation: contact Hamiltonians}) that in the contact case the Hamiltonian function $H_t \equiv 1$ generates the Reeb flow. Thus if we just define the length of a contact isotopy by the same formula that we have in the symplectic case then the length of the Reeb flow is zero (confirming the idea that the Reeb flow should be seen as a degenerate direction in the contactomorphism group). Following Banyaga and Donato \cite{BD} we can correct this problem by defining the length (relative to a contact form $\alpha$) of a contact isotopy $\{\phi_t\}_{t\in[0,1]}$ by
\begin{equation}\label{equation: MS}
l_{\alpha}(\{\phi_t\}) = \int_0^1 \max_{x\in M} H_t(x) - \min_{x \in M} H_t(x) + | c_{\alpha}(H_t) |\,dt  
\end{equation}
with 
$$
c_{\alpha}(H_t) := \frac{1}{\int_M\nu_{\alpha}} \int_M H_t \, \nu_{\alpha}
$$
where $\nu_{\alpha} = \alpha \wedge (d\alpha)^n$ is the volume form on $M$ induced by $\alpha$. Note that in the symplectic case (\ref{equation: MS}) would be just the usual formula for the Hofer norm because, since we normalize the Hamiltonians, the last term is zero.

In view of the analogy with the symplectic case, the above formula is a natural candidate to be a bi--invariant metric on the contactomorphism group. However, as we will now see, something goes wrong. In order to study the properties of the candidate distance function induced by the formula (\ref{equation: MS}) we need to look at the transformation laws for contact Hamiltonians. These are given by the following lemma.

\begin{lemma}\label{lemma: contact Hamiltonians}
Let $H_t$ and $G_t$ be contact Hamiltonian functions on $M$, generating respectively the contact isotopies $\{\phi_t\}$ and $\{\psi_t\}$. Let $h_t$ and $g_t$ be the \emph{conformal factors} of $\phi_t$ and $\psi_t$, i.e. the functions on $M$ satisfying $\phi_t^{\phantom{t}\ast}\alpha = e^{h_t} \alpha$ and $\psi_t^{\phantom{t}\ast}\alpha = e^{g_t}\alpha$. Then 
\begin{enumerate}
\item The composition $\{\phi_t \circ \psi_t\}$ is generated by the Hamiltonian function $H \,\sharp\, G$ defined by
$$\
(H \, \sharp \,G)_t = H_t + (e^{h_t} \cdot G_t) \circ (\phi_t)^{-1}\,.
$$
\item The inverse $\{\phi_t^{-1}\}$ is generated by the Hamiltonian function $\overline{H}$ defined by
$$
\overline{H}_t = - e^{-h_t} \cdot (H_t \circ \phi_t)\,.
$$
\item For a contactomorphism $\theta$ with $\theta^{\ast}\alpha = e^f \alpha$, the conjugation $\{\theta^{-1} \circ \phi_t \circ \theta\}$ is generated by the Hamiltonian function $H_{\theta}$ defined by
$$
(H_{\theta})_t = e^{-f} (H_t \circ \theta)\,.
$$
\end{enumerate}
\end{lemma}

Note that in the symplectic case we have the same formulas as in Lemma \ref{lemma: contact Hamiltonians}, but without the conformal factors. If we consider only \emph{strict} contact isotopies (as it is done in \cite{BD}), i.e. contact isotopies $\{\phi_t\}$ such that every $\phi_t$ preserves the contact form $\alpha$, then all conformal factors disappear from the above formulas. Using this it can be proved, as in the symplectic case, that (\ref{equation: MS}) induces a conjugation--invariant pseudonorm on the group $\text{Cont}(M,\alpha)$ of strict contactomorphisms. This conjugation--invariant pseudonorm has been studied by Banyaga and Donato \cite{BD} and M\"{u}ller and Spaeth \cite{MSpaeth}, and was proved by them to be always non--degenerate.

For general contact isotopies the presence of the conformal factors (thus, ultimately, the fact that contactomorphisms do not necessarily preserve the Reeb direction) perturb the proof of all properties needed to have a conjugation--invariant norm. Regarding for instance the triangle inequality, it is even possible to find examples of Hamiltonians $H$ and $G$ such that $l_{\alpha}(\{\phi_{t} \circ \psi_{t}\}) > l_{\alpha}(\{\phi_{t}\}) + l_{\alpha}(\{\psi_{t}\})$ (see \cite{MSpaeth}). In spite of this, as observed by Rybicki \cite{Rybicki} and Shelukhin \cite{Egor}, the formula (\ref{equation: MS}) does induce a pseudonorm on the contactomorphism group $\text{Cont}(M,\xi)$ and on the universal cover $\widetilde{\text{Cont}_0}(M,\xi)$ of its identity component, by taking the infimum of the lengths of all contact isotopies representing a given contactomorphism or a given element in the universal cover. For instance, even though the triangle inequality fails for the composition $\{\phi_t \circ \psi_t\}$ it can be proved to be true for the concatenation $\{\phi_t\} \sqcup \{\psi_t\}$, which is an isotopy in the same homotopy class as $\{\phi_t \circ \psi_t\}$. As proved by Shelukhin, this pseudonorm is always non--degenerate and so is a true norm on $\text{Cont}(M,\xi)$ and $\widetilde{\text{Cont}_0}(M,\xi)$. However, it is \emph{not} conjugation invariant since there is no way to neutralize the presence of the conformal factor in statement (3) of Lemma \ref{lemma: contact Hamiltonians}. We will actually see in the next section that in fact it is impossible on the contactomorphism group to have a conjugation--invariant norm which takes values arbitrarily close to zero.


\section{Rigidity and flexibility in contact topology}\label{section: Rigidity and flexibility in contact topology}

In Section \ref{section: symplectic rigidity} we have discussed three famous global rigidity results in symplectic topology: Gromov's non--squeezing theorem, the existence of non--trivial bi--invariant metrics on the Hamiltonian group (the Hofer and Viterbo metrics) and the Arnold conjecture for fixed points of Hamiltonian diffeomorphisms. In Section \ref{section: Hofer-like lengths} we have seen that, in the contact case, the same formula as for the Hofer metric does not give rise to a bi--invariant metric, because of the flexibility given by the fact that contactomorphisms do not necessarily preserve the Reeb flow. Note that the flexibility given by the Reeb direction also destroys the Arnold conjecture on fixed points, and the possibility to have some non--squeezing results in the contact Euclidean space. Indeed, in the standard contact Euclidean space $(\mathbb{R}^{2n+1},\xi_{0})$ it is possible to squeeze any given domain into an arbitrarily small one, for example by using the contactomorphism 
\begin{equation}\label{equation: squeezing}
(x,y,z) \mapsto (cx,cy,c^2z)
\end{equation}
for some $c \in \mathbb{R}$ small enough. Regarding the Arnold conjecture, note for example that the Reeb flow itself has no fixed points (for a small time), because it is the flow of a non--vanishing vector field. Using the squeezing map (\ref{equation: squeezing}) we will now show, following Polterovich \cite{Polterovich - private}, that not only the natural Hofer--like candidate for a conjugation--invariant norm on the contactomorphism group fails (as we have seen in Section \ref{section: Hofer-like lengths}) but in fact it is actually impossible on the contactomorphism group to have a conjugation--invariant norm which is \emph{fine}, i.e. takes values arbitrarily close to zero. The argument that we will reproduce mimics the one that is given by Burago, Ivanov and Polterovich \cite{BIP} to show that there are no fine conjugation--invariant norms on the diffeomorphism group of a smooth manifold. 

\begin{prop}\label{proposition: Polterovich}
For any contact manifold $(M,\xi)$, the identity component $\text{Cont}_0^{\phantom{0}c}(M,\xi)$ of the group of compactly supported contactomorphisms does not admit any conjugation--invariant norm which is fine.
\end{prop}

\begin{proof}
Let $B$ be a Darboux ball in $M$ and let $\psi_1$ and $\psi_2$ be two contactomorphisms that are supported in $B$ and do not commute, i.e. $[\psi_1,\psi_2] := \psi_1\psi_2\psi_1^{-1}\psi_2^{-1} \neq \text{id}$. Suppose by contradiction that $\lVert \, \cdot\,\rVert$ is a fine conjugation--invariant norm on $\text{Cont}_0^{\phantom{0}c}(M,\xi)$, and let $\phi$ be a contactomorphism with $\lVert \phi \rVert = \epsilon$ for $\epsilon$ arbitrarily small. Since $\phi$ is not the identity there is a small ball $B_{\phi}$ in $M$ that is displaced by $\phi$. Using the fact that the contactomorphism group is transitive and that in $(\mathbb{R}^{2n+1},\xi_{0})$ (and thus also on any Darboux ball of $M$) it is possible to squeeze any given domain into an arbitrarily small one, we can find a compactly supported contactomorphism $\theta$ of $M$ such that $\theta(B) \subset B_{\phi}$. Then $\theta\psi_1\theta^{-1}$ and $\theta\psi_2\theta^{-1}$ are supported in $B_{\phi}$. Since we assume that $\lVert \, \cdot\,\rVert$ is conjugation--invariant and $[\psi_1,\psi_2] \neq \text{id}$ we have
$$
\lVert [\theta \psi_1 \theta^{-1}, \theta \psi_2 \theta^{-1}] \rVert = \lVert \theta [\psi_1,\psi_2] \theta^{-1} \rVert  = \lVert [\psi_1,\psi_2] \rVert \neq 0\,.
$$
On the other hand we claim that
$$
\lVert [\theta \psi_1 \theta^{-1}, \theta \psi_2 \theta^{-1}] \rVert \leq 4 \lVert \phi \rVert = 4 \epsilon\,,
$$
which gives a contradiction since $\epsilon$ was chosen arbitrarily small. The above inequality is true because of the following general fact from \cite{EP - Bi-invariant metrics} and \cite{BIP}. Suppose that $a_1$ and $a_2$ are supported in a domain $\mathcal{U}$ and that $b$ displaces $\mathcal{U}$. Then for every conjugation--invariant norm $\lVert \, \cdot\,\rVert$ we have that $\lVert [a_1,a_2]\rVert \leq 4 \lVert b\rVert$. Indeed, $ba_2b^{-1}$ commutes with $a_1$ and so we have
$$
\lVert [a_1,a_2] \rVert = \lVert [a_1ba_1^{-1}b^{-1},a_2]\rVert \leq 2 \lVert [a_1,b]\rVert \leq 4 \lVert b\rVert\,.
$$
\end{proof}

As we have seen, the key in the above argument is the fact that in $(\mathbb{R}^{2n+1},\xi_{0})$ we can squeeze any given domain into an arbitrarily small one. The complete flexibility in $\mathbb{R}^{2n+1}$ should be compared on the other hand with the following surprising non--squeezing result, which was discovered in 2006 by Eliashberg, Kim and Polterovich \cite{EKP}.

\begin{thm}[\cite{EKP}]\label{theorem: EKP}
Consider the manifold $\mathbb{R}^{2n} \times S^1$ with the contact structure $\xi_0 = \text{ker} (dz + \sum_{i=1}^n x_idy_i - y_i dx_i)$. If $R_2 \leq k < R_1$ for some integer $k$ then there is no contact isotopy that squeezes $B(R_1) \times S^1$ into $B(R_2) \times S^1$.
\end{thm}

Theorem \ref{theorem: EKP} was originally proved in \cite{EKP} using methods from symplectic field theory. It was later reproved in \cite{mio 1} using generating functions and, with a similar idea, by Albers and Merry \cite{AM2} using Rabinowitz Floer homology. In \cite{EKP} it was also proved that, on the other hand, if we start with $R_1 < 1$ then it is possible to squeeze $B(R_1) \times S^1$ into $B(R_2) \times S^1$ for $R_2$ arbitrarily small. A recent result of Chiu \cite{Chiu} (and a work in progress by Fraser) treats the remaining cases, proving that in fact for any $R_2 < R_1$ with $R_2 > 1$ it is not possible to squeeze $B(R_1) \times S^1$ into $B(R_2) \times S^1$. As we will briefly discuss in Section \ref{section: orderability and the oscillation metric}, the proof of the squeezing result in \cite{EKP} for $R_1 < 1$ is related to the notion of orderability, and in particular to the existence of a positive contractible loop of contactomorphisms in $S^{2n-1}$. On the other hand, the non--squeezing theorem for integers is related to the fact that the Reeb flow associated to the standard contact structure of $\mathbb{R}^{2n} \times S^1$ is 1--periodic. As we will explain in Section \ref{section: spectral metric} (following \cite{mio 1}) the 1--periodicity gives us some control on the Reeb direction, making it possible to prove Theorem \ref{theorem: EKP}.

As we have seen in the proof of Proposition \ref{proposition: Polterovich}, the fact that the contactomorphism group does not admit any fine conjugation--invariant norm is due to the possibility to squeeze any given domain of $\mathbb{R}^{2n+1}$ into an arbitrarily small one. On the other hand, Theorem \ref{theorem: EKP} suggests that, at least for certain contact manifolds such as for example $\mathbb{R}^{2n} \times S^1$, it might be possible to have \emph{discrete} conjugation--invariant norms on the contactomorphism group. In the next section we will see that this is indeed the case. Following \cite{mio 1} we will show how to construct an integer--valued bi--invariant metric on the group of compactly supported contactomorphisms of $\mathbb{R}^{2n} \times S^1$. This metric is a natural contact analogue of the Viterbo metric on the group of compactly supported Hamiltonian symplectomorphisms of $\mathbb{R}^{2n}$. As we will see in the next section, while the Viterbo metric is defined in terms of the symplectic action of fixed points of Hamiltonian symplectomorphisms, the metric on $\mathbb{R}^{2n} \times S^1$ is related to the \emph{contact action} (or \emph{time--shift}) of \emph{translated points} of contactomorphisms.

Roughly speaking, translated points can be thought of as fixed points modulo the Reeb flow. The precise definition is as follows. Let $\big(M,\xi = \text{ker}(\alpha)\big)$ be a contact manifold and $\phi$ a contactomorphism. A point $q$ of $M$ is called a \textbf{translated point} of $\phi$, with respect to the contact form $\alpha$, if $\phi(q)$ and $q$ are in the same Reeb orbit and $\phi^{\ast}\alpha_q = \alpha_q$. As mentioned above, and as we will see in the next section, translated points of contactomorphisms play the same role in the definition of the bi--invariant metric on the contactomorphism group of $\mathbb{R}^{2n} \times S^1$ as the one that is played in the symplectic case by fixed points of Hamiltonian symplectomorphisms. A crucial difference between translated points and fixed points of Hamiltonian diffeomorphisms, that explains the different kind of results that we have in the contact case, is that translated points are flexible, in particular in the sense that they are not invariant by conjugation. In the symplectic case, fixed points of Hamiltonian symplectomorphisms and their symplectic action are invariant by conjugation. Indeed if $q$ is a fixed point of a Hamiltonian symplectomorphism $\varphi$ then for any other $\psi$ we have that $\psi(q)$ is a fixed point of the conjugation $\psi\varphi\psi^{-1}$. Moreover the symplectic action of $q$ as a fixed point of $\varphi$ coincides to the symplectic action of $\psi(q)$ as a fixed point of $\psi \varphi \psi^{-1}$ (see for example \cite[5.2]{HZ}). In the contact case, if $q$ is a translated points of a contactomorphism $\phi$ then there is no reason in general why $\psi(q)$ should be a translated point of $\psi\phi\psi^{-1}$. Indeed in general $\psi$ does not preserve the Reeb flow and so the fact that $q$ and $\phi(q)$ lie on the same Reeb orbit does not necessarily imply that the same should be true also for the images $\psi(q)$ and $\psi\phi\psi^{-1}\big(\psi(q)\big) = \psi\big(\phi(q)\big)$. Note however that translated points that are also fixed points (in other words, fixed points $q$ of $\phi$ with $\phi^{\ast}\alpha_q = \alpha_q$) are invariant by conjugation (indeed, an easy calculation shows that if $\phi^{\ast}\alpha_q = \alpha_q$ then $(\psi\phi\psi^{-1})^{\ast} \alpha_{\psi(q)} = \alpha_{\psi(q)}$). Translated points that are also fixed points are called \textbf{discriminant points}. As we will see, the fact that translated points are flexible while discriminant points are rigid makes that the contact analogue of the Viterbo metric does not work on $\mathbb{R}^{2n+1}$ but only on $\mathbb{R}^{2n} \times S^1$, and there it becomes integer--valued. Indeed, as we will see, in order to be bi--invariant this metric is defined in terms of discriminant points, that on $\mathbb{R}^{2n} \times S^1$ correspond to translated points with \emph{integer} time--shift. Note that a very similar phenomenon appears also in the work of Givental on the non--linear Maslov index \cite{Givental - Nonlinear Maslov index}, where the rigidity given by the discriminant points of contactomorphisms is used to define a quasimorphism on the universal cover of the contactomorphism group of real projective space $\mathbb{R}P^{2n-1}$ (with the contact structure induced by the standard one on the sphere $S^{2n-1}$). 

In Section \ref{section: discriminant metric} we will see how to use the rigidity given by discriminant points to define a bi--invariant metric (the \emph{discriminant metric}) on the universal cover of the contactomorphism group of any contact manifold. In Section \ref{section: orderability and the oscillation metric} the discriminant metric will then be combined with the notion of orderability to obtain what we call the \emph{oscillation metric}.

To conclude this section, note that translated points not only, as we will see, play the same role as fixed points of Hamiltonian diffeomorphisms in the definition of the spectral metric on $\mathbb{R}^{2n} \times S^1$ but, as first observed in \cite{mio 4, mio 5}, they also seem to satisfy an analogue of the Arnold conjecture on fixed points of Hamiltonian diffeomorphisms. Indeed, let $\phi$ be a contactomorphism of $\big(M,\xi = \text{ker}(\alpha)\big)$ with $\phi^{\ast}\alpha = e^g \alpha$ and consider its \emph{contact graph}
$$
\text{gr}(\phi) = \{\,\big(q,\phi(q),g(q)\big)\;,\; q \in M\,\}
$$
in the \emph{contact product} $\big(M \times M \times \mathbb{R}, \text{ker}(e^{\theta} \alpha_1 - \alpha_2)\big)$, where $\theta$ is the coordinate in $\mathbb{R}$ and $\alpha_1$ and $\alpha_2$ the pullback of $\alpha$ by the projection on the first and second factor. We have that $\text{gr}(\phi)$ is a \emph{Legendrian submanifold} of $M \times M \times \mathbb{R}$, i.e. an integral submanifold of maximal dimension of the contact distribution. While in the symplectic case fixed points of a Hamiltonian symplectomorphism correspond to intersections between its graph and the diagonal, now we have that translated points of a contactomorphism $\phi$ are in 1--1 correspondence with Reeb chords between the graph $\text{gr}(\phi)$ and the diagonal $\Delta:= \text{gr}(\text{id})$. By the contact version of the Weinstein theorem we know that the diagonal $\Delta$ in $M \times M \times \mathbb{R}$ has a neighborhood which is strictly contactomorphic to a neighborhood of the 0--section in $J^1\Delta$. If we assume that the contactomorphism $\phi$ is $\mathcal{C}^1$--small then its graph $\text{gr}(\phi)$ is contained in this tubular neighborhood and in fact corresponds to a Legendrian section of $J^1\Delta$, and hence to the 1--jet 
$$
j^1f = \{\,\big(x,df(x),f(x)\big) \; , \; x \in \Delta\,\}
$$
of a function $f$ on $\Delta$. Critical points of $f$ correspond to Reeb chords between $j^1f$ and the 0--section, hence to Reeb chords between $\text{gr}(\phi)$ and $\Delta$, hence to translated points of $\phi$. Thus, every $\mathcal{C}^1$--small contactomorphism of a compact contact manifold $\big(M,\xi = \text{ker}(\alpha)\big)$ always has translated points, at least as many as the minimal number of critical points of a smooth function on $M$. In analogy with the Arnold conjecture in the symplectic case, it seems natural to conjecture that the same should be true for all contactomorphisms that are contact isotopic to the identity. Since it was proposed in \cite{mio 4, mio 5}, this problem has been studied and proved in some special cases by several authors \cite{mio 5, AM1, AM3, Egor, Meiwes - Naef, Zenaidi}. In a work in progress \cite{Floer homology for translated points} I am constructing a Floer homology theory for translated points in order to obtain a good framework to study this conjecture for more general contact manifolds. In view of the geometric relevance of translated and discriminant points I then expect that this homology theory will be also a good tool to study the contact rigidity phenomena that are discussed in this article, and so in particular bi--invariant metrics on the contactomorphism group.


\section{The spectral metric on $\mathbb{R}^{2n} \times S^1$}\label{section: spectral metric}

In this section we will review how the Viterbo metric on the group of compactly supported Hamiltonian symplectomorphisms of $\mathbb{R}^{2n}$ is defined and then show, following \cite{mio 1, mio 2}, how the construction can be generalized to obtain an integer valued bi--invariant metric on the identity component of the contactomorphism group of $\mathbb{R}^{2n} \times S^1$. 

As already mentioned above, the Viterbo metric is defined using the method of \emph{generating functions}. Generating functions are smooth real functions, defined on some finite dimensional manifolds,  that are associated to certain Lagrangian submanifolds of cotangent bundles. The most basic example of a generating function is given by the fact that if $f: B \rightarrow \mathbb{R}$ is a smooth function on a smooth manifold $B$ then the image $L_f$ of the differential $df: B \rightarrow T^{\ast}B$ is a Lagrangian submanifold. In this case we say that $f$ is a generating function for $L_f$. Note that critical points of $f$ correspond to intersections of $L_f$ with the 0--section, and in fact the whole geometry of the Lagrangian submanifold $L_f$ is determined by the Morse theory of $f$. In order to associate a function to more general Lagrangian submanifolds of $T^{\ast}B$ we can extend the above idea as follows. Instead of considering just functions defined on $B$ we consider, more generally, functions defined on the total space of a fiber bundle over $B$. Then, instead of taking the graph of the differential we take the graph of the horizontal derivative along fiber critical points. This construction goes back to H\"{o}rmander \cite{Hor}, and in more details can be described as follows. Let $F: E \rightarrow \mathbb{R}$ be a function defined on the total space of a fiber bundle $p: E \rightarrow B$. Assume that $dF: E \rightarrow T^{\ast}E$ is transverse to the fiber normal bundle 
$$
N_E := \{\,(e,\mu)\in T^{\ast}E\,|\,\mu=0 \text{ on } \text{ker}\,dp(e)\,\}
$$
so that the set $\Sigma_F := dF^{-1}(N_E)$ of fiber critical points is a submanifold of $E$, of dimension equal to the dimension of $B$. Given a point $e$ of $\Sigma_F$ we associate to it an element $v^{\ast}(e)$ of $T_{p(e)}^{\phantom{e}\ast}B$ by defining
$$
v^{\ast}(e) = dF(\widehat{X})
$$
for $X\in T_{p(e)}B$, where $\widehat{X}$ is any vector in $T_eE$ with $p_{\ast}(\widehat{X}) = X$. The map 
$$i_F: \Sigma_F \rightarrow T^{\ast}B$$
defined by $e \mapsto \big(p(e),v^{\ast}(e)\big)$ is a Lagrangian immersion, and the function $F: E \rightarrow \mathbb{R}$ is said to be a generating function for the (possibly immersed) Lagrangian submanifold $L_F = i_F(\Sigma_F)$ of $T^{\ast}B$.  Note that if $E=B$ and $p$ is the identity then $i_F: \Sigma_F \rightarrow T^{\ast}B$ is just the graph of the differential of $F$. Note also that, as in the special case of the differential of a function, critical points of a generating function $F$ are in 1--1 correspondence with intersections of the Lagrangian submanifold $L_F$ with the 0--section. By the work of Chaperon, Laudenbach and  Sikorav \cite{Ch1, LS, Sikorav - sur les immersions lagrangiennes, Sikorav - problemes d'intersections} we know that, if $B$ is closed, any Lagrangian submanifold $L$ of $T^{\ast}B$ that is Hamiltonian isotopic to the 0--section has a generating function $F_L: B \times \mathbb{R}^N \rightarrow \mathbb{R}$ quadratic at infinity. Moreover, if $\varphi_t$ is a Hamiltonian isotopy starting at the identity then there is a smooth 1--parameter family $F_t$ of generating functions quadratic at infinity for the Lagrangian isotopy given by the image of the 0--section by $\varphi_t$. As proved by Viterbo and Th\'{e}ret \cite{Viterbo, Th3}, generating functions quadratic at infinity of Lagrangian submanifolds of $T^{\ast}B$ which are Hamiltonian isotopic to the 0--section are essentially unique.

In \cite{Viterbo} Viterbo applied the construction of generating functions to the case of compactly supported Hamiltonian symplectomorphisms of $\mathbb{R}^{2n}$ as follows. Let $\varphi$ be a Hamiltonian symplectomorphism of $\mathbb{R}^{2n}$. We identify the product $\overline{\mathbb{R}^{2n}}\times\mathbb{R}^{2n}$ with the cotangent bundle $T^{\ast}\mathbb{R}^{2n}$ by the symplectomorphism
$$
\tau(x,y, X, Y) = (\frac{x+X}{2},\frac{y+Y}{2}, Y-y,x-X)
$$
and consider the Lagrangian submanifold $\Gamma_{\varphi}$  of $T^{\ast}\mathbb{R}^{2n}$ that corresponds to the graph of $\varphi$. Since we assume that $\varphi$ is compactly supported, and since $\tau$ sends the diagonal to the 0--section, we have that $\Gamma_{\varphi}$ coincides with the 0--section outside a compact set. So $\Gamma_{\varphi}$ can be seen as a Lagrangian submanifold of $T^{\ast}S^{2n}$, by seeing $S^{2n}$ as the 1--point compactification of $\mathbb{R}^{2n}$. Note also that $\Gamma_{\varphi}$ is Hamiltonian isotopic to the 0--section, because $\varphi$ is Hamiltonian isotopic to the identity. By the existence results for generating function  of Lagrangians we obtain thus a generating function quadratic at infinity $F_{\varphi}: S^{2n} \times \mathbb{R}^N \rightarrow \mathbb{R}$ for $\Gamma_{\varphi}$. Note that critical points of $F_{\varphi}$ correspond to intersections of $\Gamma_{\varphi}$ with the 0--section and hence to fixed points of $\varphi$. Moreover it can be shown that the critical values of the generating function $F_{\varphi}$ coincide with the symplectic action of the corresponding fixed points of $\varphi$.

By applying a classical minimax method to the generating function $F_{\varphi}$ of a compactly supported Hamiltonian symplectomorphism $\varphi$ of $\mathbb{R}^{2n}$, Viterbo \cite{Viterbo} used the unit and top cohomology classes in $S^{2n}$ to select critical values of $F_{\varphi}$ and define \emph{spectral invariants} $c^{\pm}(\varphi)$. As proved in \cite{Viterbo}, these numbers satisfy the following properties.

\begin{prop}\label{proposition: properties of spectral invariants for symplectomorphisms}
\begin{enumerate}
\item The maps $\text{Ham}^c(\mathbb{R}^{2n}) \rightarrow \mathbb{R}$, $\varphi \mapsto c^{\pm}(\varphi)$ are continuous.
\item $c^+(\varphi) \geq 0$ and $c^-(\varphi)\leq 0$.
\item $c^+(\varphi) = c^-(\varphi) = 0$ if and only if $\varphi$ is the identity.
\item $c^-(\varphi) = - \, c^+(\varphi^{-1})$.
\item $c^+(\varphi\psi) \leq c^+(\varphi) + c^+(\psi)$.
\item $c^{\pm}(\psi\varphi\psi^{-1}) = c^{\pm}(\varphi)$.
\end{enumerate}
\end{prop}

For a compactly supported Hamiltonian symplectomorphism $\varphi$ we now define its \emph{Viterbo norm} by
$$
\lVert \varphi \rVert_V = c^+(\varphi) - c^-(\varphi)\,.
$$
Using Proposition \ref{proposition: properties of spectral invariants for symplectomorphisms} it is easy to prove that $\lVert \cdot \rVert_V$ is a conjugation--invariant norm on $\text{Ham}^c(\mathbb{R}^{2n})$.

The spectral invariant $c^+$ can also be used to define a \emph{symplectic capacity} for domains. Given an open bounded domain $\mathcal{U}$, we define $c(\mathcal{U})$ to be the supremum of the values of $c^{+}(\varphi)$ for all $\varphi$ supported in $\mathcal{U}$. Note that this is well--defined because any bounded domain of $\mathbb{R}^{2n}$ can be displaced, and it can be proved that if $\varphi$ is supported in $\mathcal{U}$ and $\psi$ displaces $\mathcal{U}$ then 
\begin{equation}\label{equation: displacement}
c^+(\varphi) \leq \lVert \psi\rVert_V\,.
\end{equation}
We then extend the definition of the capacity to arbitrary domains in $\mathbb{R}^{2n}$ as follows. For an open domain $\mathcal{V}$ we define $c(\mathcal{V})$ as the supremum of all values of $c(\mathcal{U})$ for $\mathcal{U} \subset \mathcal{V}$ bounded. For an arbitrary domain $A$ we then define $c(A)$ to be the the infimum of the values of $c(\mathcal{V})$ for $\mathcal{V}$ open containing $A$. Because of property (6) in Proposition \ref{proposition: properties of spectral invariants for symplectomorphisms} we have that $c$ is a symplectic invariant, i.e. $c\big(\psi(\mathcal{U})\big) = c(\mathcal{U})$ for all domains $\mathcal{U}$ of $\mathbb{R}^{2n}$ and Hamiltonian symplectomorphisms $\psi$. Moreover we have by definition that $c$ is monotone, i.e. if $\mathcal{U} \subset \mathcal{V}$ then $c(\mathcal{U}) \leq c(\mathcal{V})$. Gromov's non--squeezing theorem follows then from the fact, proved in \cite{Viterbo}, that 
$$
c\big(B^{2n}(R)\big) = c\big(C^{2n}(R)\big) = R\,.
$$
Note also that, if we define the \emph{displacement energy} of a domain $\mathcal{U}$ as the infimum of the norms $\lVert\psi\rVert_V$ for all $\psi$ displacing $\mathcal{U}$, then (\ref{equation: displacement}) implies the \emph{energy--capacity inequality} $c(\mathcal{U}) \leq E(\mathcal{U})$.

In the contact case, generating functions are associated to Legendrian submanifolds of the 1--jet bundle $J^1B$ of a smooth manifold $B$. Given a smooth function $f$ on $B$, its 1--jet $j^1f = \{\, \big(x,df(x),f(x)\big)\, , \, x \in B\,\}$ is a Legendrian submanifold of $J^1B$. More generally, if $F: E \rightarrow \mathbb{R}$ is a function defined on the total space of a fiber bundle over $B$, such that $dF$ is transverse to the fiber normal bundle, then the map
$$
j_F: \Sigma_F \rightarrow J^1B
$$ 
(where, as above, $\Sigma_F$ denotes the space of fiber critical points) which is defined by $$e \mapsto \big(p(e),v^{\ast}(e),F(e)\big)$$ is a Legendrian immersion. The function $F$ is said to be a generating function for the Legendrian submanifold $\widetilde{L_F} = j_F(\Sigma_F)$ of $J^1B$. Critical points of $F$ correspond now not to intersections between $\widetilde{L_F} $ and the 0--section but to Reeb chords between them (intersections correspond to critical points of critical value zero). It was proved independently by Chekanov \cite{C} and Chaperon \cite{Ch2} that, if $B$ is compact, for every Legendrian isotopy in $J^1B$ starting at the 0--section there is a smooth 1--parameter family of generating functions $F_t: B \times \mathbb{R}^{N} \rightarrow \mathbb{R}$ quadratic at infinity. Moreover, as in the symplectic case we know that generating functions quadratic at infinity for Legendrian submanifolds isotopic to the 0--section are essentially unique \cite{Th1, Th3}. Following Bhupal \cite{Bhupal}, and in analogy with the work of Viterbo, we can now associate a generating function quadratic at infinity to every compactly supported contactomorphism $\phi$ of $\mathbb{R}^{2n+1}$ which is contact isotopic to the identity. Indeed we can identify the contact product $\mathbb{R}^{2n+1} \times \mathbb{R}^{2n+1} \times \mathbb{R}$ with the 1--jet bundle $J^1\mathbb{R}^{2n+1}$ by a map that sends the diagonal to the 0--section (see \cite{Bhupal} or \cite{mio 1, mio 3} for explicit formulas). We can then consider the Legendrian submanifold $\Gamma_{\phi}$ in $J^1\mathbb{R}^{2n+1}$ that corresponds to the graph of $\phi$. If we assume $\phi$ to be compactly supported then $\Gamma_{\phi}$ can be seen as a Legendrian submanifold of $J^1S^{2n+1}$. Thus, as in the symplectic case, we can consider a generating function quadratic at infinity for $\phi$ and define by minimax spectral numbers $c^{\pm}(\phi)$. However in the contact case these numbers can not be used to obtain the same applications as in the work of Viterbo, because now the analogues of properties (4), (5) and (6) of Proposition \ref{proposition: properties of spectral invariants for symplectomorphisms} do not hold anymore. 

Let's see in particular why in the contact case the numbers $c^{\pm}$ are not invariant by conjugation. In the symplectic case invariance by conjugation of $c^{\pm}$ follows from the fact that critical points of the generating function of a Hamiltonian diffeomorphism $\varphi$ correspond to fixed points of $\varphi$ (with critical value given by the symplectic action), and fixed points of Hamiltonian symplectomorphisms (and their symplectic action) are invariant by conjugation. The precise argument goes as follows. Let $\varphi$ and $\psi$ be Hamiltonian diffeomorphisms of $\mathbb{R}^{2n}$. We want to show that $c^{\pm}(\varphi) = c^{\pm}(\psi\varphi\psi^{-1})$. Let $\psi_t$ be a Hamiltonian isotopy joining $\psi$ to the identity. For every $t$, the numbers $c^{\pm}(\psi_t\varphi\psi_t^{-1})$ are in the action spectrum of $\psi_t\varphi\psi_t^{-1}$. Since the action spectrum is discrete and invariant by conjugation, and since the maps $t \mapsto c^{\pm}(\psi_t\varphi\psi_t^{\phantom{t}-1})$ are continuous, we conclude in particular that $c^{\pm}(\psi\varphi\psi^{-1}) = c^{\pm}(\varphi)$. In the contact case we have that critical points of the generating function of a contactomorphism $\phi$ of $\mathbb{R}^{2n+1}$ correspond to Reeb chords between $\Gamma_{\phi}$ and the 0--section in $J^1\mathbb{R}^{2n+1}$, hence to Reeb chords between the graph of $\phi$ and the diagonal, hence to translated points of $\phi$. Moreover it can be proved that the critical value of the generating function is equal to the \emph{time--shift} of the corresponding translated point $q$ of $\phi$, i.e. to the difference in the $z$--direction between $q$ and $\phi(q)$. As discussed in Section \ref{section: Rigidity and flexibility in contact topology}, translated points are not invariant by conjugation. Indeed, if $q$ is a translated point of $\phi$ then $\psi(q)$ is in general not a translated point of $\psi\phi\psi^{-1}$ because $\psi$ does not necessarily preserve the Reeb flow. It is at this point that the proof of conjugation--invariance breaks. For similar reasons, the proofs of properties (4) and (5) of Proposition \ref{proposition: properties of spectral invariants for symplectomorphisms} also fail in the contact case. By the discussion in Section \ref{section: Rigidity and flexibility in contact topology} we know that this failure is not just technical, since it is actually impossible in the contact Euclidean space $\mathbb{R}^{2n+1}$ to have the same applications as in the work of Viterbo.

We will now see how, in spite of what we discussed in the previous paragraph, some global rigidity as in the work of Viterbo survives if, instead of $\mathbb{R}^{2n+1}$, we consider the contact manifold $\mathbb{R}^{2n} \times S^1$. We will work with the contact form\footnote{In order to be able to use generating functions we need to work with the standard contact form of $\mathbb{R}^{2n} \times S^1$. However, as we will see, the objects that we will obtain (a contact capacity for domains and a bi--invariant metric on the contactomorphism group) will be invariant by all contactomorphisms (isotopic to the identity), not just those that preserve the chosen contact form.} on $\mathbb{R}^{2n} \times S^1$  which is induced by the standard contact form of $\mathbb{R}^{2n+1}$. The associated Reeb flow is then given by rotations in the $S^1$--direction, and so in particular is 1--periodic. We will see compactly supported contactomorphisms of $\mathbb{R}^{2n} \times S^1$ as contactomorphisms of $\mathbb{R}^{2n+1}$ that are 1--periodic in the $z$--direction and compactly supported in the $(x,y)$--plane. The crucial observation is then that if $\phi$ is a 1--periodic contactomorphism of $\mathbb{R}^{2n+1}$ then translated points of $\phi$ with \emph{integer} time--shift are also fixed points. Thus for the group of 1--periodic contactomorphisms of $\mathbb{R}^{2n+1}$ (compactly supported in the $(x,y)$--plane) translated points with integer time--shift are invariant by conjugation. More precisely we have the following result. Suppose that $k$ is an integer and that $\phi$ and $\psi$ are 1--periodic contactomorphisms of $\mathbb{R}^{2n+1}$. Then a point $q$ of $\mathbb{R}^{2n+1}$ is a translated point of $\phi$ with time--shift $k$ if and only if $\psi(q)$ is a translated point of $\psi\phi\psi^{-1}$ with time--shift $k$. This basic fact is the key ingredient to prove the following result.

\begin{prop}[\cite{mio 1, mio 2}]\label{proposition: properties of spectral invariants for contactomorphisms}
For all contactomorphisms of $\mathbb{R}^{2n+1}$ that are 1--periodic in the $z$--direction and compactly supported in the $(x,y)$--plane (and are isotopic to the identity through contactomorphisms of this form) the following properties holds:
\begin{enumerate}
\item $\lfloor c^-(\phi)\rfloor = - \lceil c^+(\phi^{-1})\rceil$
\item $\lceil c^+(\phi\psi)\rceil \leq \lceil c^+(\phi)\rceil + \lceil c^+(\psi)\rceil$
\item $\lceil c^{\pm}(\phi)\rceil = \lceil c^{\pm}(\psi\phi\psi^{-1})\rceil$ and $\lfloor c^{\pm}(\phi)\rfloor = \lfloor c^{\pm}(\psi\phi\psi^{-1})\rfloor$.
\end{enumerate}
\end{prop}

Following an argument that was originally given by Bhupal \cite{Bhupal} for the case $k=0$, we now explain in particular why (3) is true. Let $\psi_t$ be a contact isotopy (of 1--periodic contactomorphisms) from the identity to $\psi_1 = \psi$, and let $F_t: E \rightarrow \mathbb{R}$ be a 1--parameter family of generating functions for $\psi_t\phi\psi_t^{-1}$. The idea is to study the bifurcation diagram of the 1--parameter family $F_t$ and show that there can be no path $c_t$ of critical values for $F_t$ which crosses an integer $k$. Indeed, suppose that $c_t$ is a path of critical values of $F_t$ with $c_{t_0} = k$ for some $t_0$. Then we claim that $c_t = k$ for all $t$. This can be seen as follows. Let $x_t$ be a path in $E$ such that, for every $t$ in a subinterval of $[0,1]$ containing $t_0$, $x_t$ is a critical point of $F_t$ of critical value $c_t$. Each $x_t$ corresponds to a translated point of $\psi_t\phi\psi_t^{-1}$ with time--shift $c_t$. In particular, $x_{t_0}$ corresponds to a discriminant point of $\psi_{t_0}\phi\psi_{t_0}^{-1}$. Assume first that $x_{t_0}$ is a non--degenerate critical point of $F_{t_0}$. The idea of the proof now is to construct a path $y_t$ in $E$ such that $y_{t_0} = x_{t_0}$ and each $y_t$ is a non--degenerate critical point of $F_t$ with critical value $k$. It will then follow from Morse theory that the two paths $x_t$ and $y_t$ must coincide, so that $c_t = k$ for all $t$. The path $y_t$ can be constructed as follows. Let $q_{t_0} \in \mathbb{R}^{2n+1}$ be the translated point of $\psi_{t_0}\phi\psi_{t_0}^{-1}$ (with time--shift $k$) corresponding to the critical point $x_{t_0}$ of $F_{t_0}$. By the discussion above we have that, for every $t$, $\psi_t\big(\psi_{t_0}^{-1}(q_{t_0})\big)$ is a translated point of $\psi_{t}\phi\psi_{t}^{-1}$ with time--shift $k$. We then take $y_t$ to be the critical point of $F_t$ corresponding to this translated point. Since we assume that $x_{t_0}$ is non--degenerate, it follows from the construction that all $y_t$ are non--degenerate (see \cite{Bhupal, mio 1} for more details). This finishes the proof in the case when $x_{t_0}$ is non--degenerate. The general case follows then from an approximation argument. Properties (1) and (2) of Proposition \ref{proposition: properties of spectral invariants for contactomorphisms} can also be proved by similar arguments, combined with the proofs given by Viterbo of the corresponding properties in the symplectic case.

It follows from Proposition \ref{proposition: properties of spectral invariants for contactomorphisms} (and the fact that the other statements in Proposition \ref{proposition: properties of spectral invariants for symplectomorphisms} go through without changes in the contact case) that the definition of the Viterbo bi--invariant metric and of the Viterbo capacity for domains can be given also in the contact case for $\mathbb{R}^{2n} \times S^1$, but become integer--valued. For a compactly supported contactomorphism $\phi$ of $\mathbb{R}^{2n} \times S^1$ isotopic to the identity we namely obtain an \textbf{integer--valued conjugation--invariant norm} by defining
$$
\lVert \phi \rVert = \lceil c^+(\phi)\rceil - \lfloor c^-(\psi)\rfloor\,.
$$
We also obtain an integer--valued contact invariant for domains of $\mathbb{R}^{2n} \times S^1$ by defining $c(\mathcal{U})$ to be the supremum of all values of  $\lceil c^+(\phi)\rceil$ for $\phi$ supported in $\mathcal{U}$. As in the symplectic case this is well--defined because if $\psi$ is a contactomorphism of $\mathbb{R}^{2n} \times S^1$ that displaces $\mathcal{U}$ then it can be proved that 
\begin{equation}\label{equation: displacement contact}
\lceil c^+(\phi)\rceil \leq \lVert \psi \rVert\,.
\end{equation}
Because of property (3) in Proposition \ref{proposition: properties of spectral invariants for contactomorphisms} the integer number $c$ is indeed a contact invariant, i.e. $c\big(\psi(\mathcal{U})\big) = c(\mathcal{U})$ for all domains $\mathcal{U}$ and contactomorphisms $\psi$ (isotopic to the identity). Since it can be proved that, for a domain $\mathcal{U}$ of $\mathbb{R}^{2n}$, $c(\mathcal{U} \times S^1) = \lceil c(\mathcal{U})\rceil$, the contact non-squeezing theorem of Eliashberg, Kim and Polterovich follows. As in the symplectic case we can define the \emph{displacement energy} of a domain $\mathcal{U}$ of $\mathbb{R}^{2n} \times S^1$ as the infimum of all values $\lVert \psi \rVert$ for all $\psi$ displacing $\mathcal{U}$. By (\ref{equation: displacement contact}) we then have the \textit{energy--capacity inequality} $c(\mathcal{U}) \leq E(\mathcal{U})$.

\begin{rmk}
In \cite{Viterbo} Viterbo defined also a bi--invariant partial order on the Hamiltonian group of $\mathbb{R}^{2n}$, by posing $\phi_1 \leq \phi_2$ if $c^+(\phi_1\phi_2^{-1}) = 0$. In \cite{Bhupal} Bhupal proved that the analogous definition gives rise to a bi--invariant partial order also on the group of compactly supported contactomorphisms of $\mathbb{R}^{2n+1}$. In order to have bi--invariance of the partial order Bhupal proved, by the argument that we have reproduced above to prove property (3) of Proposition \ref{proposition: properties of spectral invariants for contactomorphisms}, that $c^+(\phi) = 0$ if and only if $c^+(\psi\phi\psi^{-1}) = 0$. Note that the argument that we have used above works for $k=0$ also in $\mathbb{R}^{2n+1}$, because in the case $k=0$ we only need to know that translated points of time--shift zero are fixed points (hence invariant by conjugation) which is true also in $\mathbb{R}^{2n+1}$. Of course in the case of $\mathbb{R}^{2n} \times S^1$ we can also define a bi--invariant partial order in the same way. It is proved in \cite{mio 2} that this partial order is compatible with the bi--invariant metric, in the sense that if $\phi_1 \leq \phi_2 \leq \phi_3$ then $d(\phi_1,\phi_2) \leq d(\phi_1,\phi_3)$.
\end{rmk}

As discussed in \cite{mio 2}, the bi--invariant metric that we defined on the contactomorphism group of $\mathbb{R}^{2n} \times S^1$ is (as the Viterbo metric) unbounded but not stably unbounded. Unboundedness is proved by constructing explicitly Hamiltonian diffeomorphisms of $\mathbb{R}^{2n}$ with Viterbo norm arbitrarily large, and then lifting these Hamiltonian diffeomorphisms to $\mathbb{R}^{2n} \times S^1$ to obtain contactomorphisms of arbitrarily large norm. More precisely this can be done as follows. Let $H: \mathbb{R}^{2n} \rightarrow \mathbb{R}$ be the function
$$
H(x_1,y_1, \cdots,x_n,y_n) = \sum_{i=1}^n \frac{\pi}{R} (x_i^2 + y_i^2)
$$
for some $R$ that will be chosen arbitrarily big, and consider $H_{\rho} = \rho \circ H$ where $\rho: [0,\infty) \rightarrow [0,\infty)$ is a function supported in $[0,1]$ with $\rho'' > 0$ in the interior of the support. Take a sequence $\rho_1$, $\rho_2$, $\rho_3$, $\cdots$ of functions of this form, with $\lim_{i \to \infty} \rho_i(0) = \infty$ and $\lim_{i \to \infty} \rho'_i(0) = - \infty$ and such that $H_{\rho_1} \leq H_{\rho_2} \leq H_{\rho_3} \leq \cdots$ with $H_{\rho_i}$ getting pointwise arbitrarily large on $B^{2n}(R)$. Let $\varphi^{\rho_i}$ be the time--1 map of the Hamiltonian flow of $H_{\rho_i}$. Since the $H_{\rho_i}$ are non--negative it can be shown that $c^-(\varphi^{\rho_i}) = 0$ and so $\lVert \varphi^{\rho_i}\rVert_V = c^+(\varphi^{\rho_i})$. Moreover it is proved by Traynor \cite{Traynor} that $c^+(\varphi^{\rho_i})$ tends to $R$ as $i \rightarrow \infty$. Thus by choosing $R$ big enough we can obtain in this way Hamiltonian diffeomorphisms of $\mathbb{R}^{2n}$ of arbitrarily large Viterbo norm. Note now that any compactly supported Hamiltonian diffeomorphism $\varphi$ of $\mathbb{R}^{2n}$ can be lifted to a compactly supported contactomorphism $\widetilde{\varphi}$ of $\mathbb{R}^{2n} \times S^1$, by defining
$$
\widetilde{\varphi} (x,y,z) = \big(\varphi_1(x,y), \varphi_2(x,y), z + F(x,y)\big)
$$
where $\varphi = (\varphi_1,\varphi_2)$ and $F$ is the compactly supported function satisfying $\varphi^{\ast}(\lambda_0) - \lambda_0 = dF$. It was proved in \cite{mio 1} that, for any $\varphi$, $c^+(\widetilde{\varphi}) = c^+(\varphi)$ and $c^-(\widetilde{\varphi}) = c^-(\varphi)$. We can thus obtain compactly supported contactomorphisms of $\mathbb{R}^{2n} \times S^1$ with arbitrarily large norm.

Recall that an unbounded norm $\lVert \cdot \rVert$ on a group $G$ is said to be \emph{stably unbounded} if there is an element $f$ of $G$ such that $\lim_{n \to \infty} \frac{\lVert f^n\rVert}{n} \neq 0$. In our case we have that $\lim_{n \to \infty} \frac{\lVert \phi^n\rVert}{n} = 0$ for all contactomorphisms $\phi$ of $\mathbb{R}^{2n} \times S^1$ because all $\phi^n$ have the same support and, by the definition of the capacity, the norm of a contactomorphism is never bigger than twice the capacity of its support. Thus our bi--invariant metric on the contactomorphism group of $\mathbb{R}^{2n} \times S^1$ is not stably unbounded.

In this section we have seen that, at least in the case of Euclidean space and $\mathbb{R}^{2n} \times S^1$, translated points are the crucial objects to look at in order to study the applications that we are discussing in this paper. Indeed these applications are obtained by doing Morse theory for the generating functions, and translated points are the points that correspond to the critical points of the functions. With respect to fixed points of Hamiltonian symplectomorphisms, translated points have the problem of being flexible, in particular not invariant by conjugation. However translated points that are also fixed points (i.e. discriminant points) are rigid. In $\mathbb{R}^{2n} \times S^1$ it is the rigidity given by discriminant points that allowed us to define a non--trivial bi--invariant metric on the contactomorphism group.

In the next section we will see how to use this idea to define a bi--invariant metric on the universal cover of the contactomorphism group  of any contact manifold.

\begin{rmk}
Mimicking the construction in \cite{mio 2}, Zapolsky \cite{Zapolsky} defined an integer--valued bi--invariant metric on the contactomorphism group of $T^{\ast}B \times S^1$ for any closed manifold $B$. For a contactomorphism $\phi$ of $T^{\ast}B \times S^1$ he considers a generating function quadratic at infinity $F_{\phi}$ for the image $L_{\phi}$ of the 0--section by $\phi$. Note that $F_{\phi}$ does not see the whole $\phi$ but only what $\phi$ does to the 0--section. In particular, only translated points of $\phi$ that are on the 0--section are seen by $F_{\phi}$ (and on the other hand there are critical points of $F_{\phi}$ that do not correspond to translated points of $\phi$, since a Reeb chord between the 0--section and $L_{\phi}$ does not necessarily join a point $x$ of the 0--section to its image $\phi(x)$). However it was proved in \cite{Zapolsky} that if we define
$$
\lVert \phi \rVert = \max_{\psi} \{\,\lceil c^+(\psi\phi\psi^{-1})\rceil - \lfloor c^-(\psi\phi\psi^{-1})\rfloor\,\}
$$
then we obtain a (stably unbounded) conjugation--invariant norm. Note that all translated points of $\phi$ with integer time--shift appear as critical points of some function $F_{\psi\phi\psi^{-1}}$. However the functions $F_{\psi\phi\psi^{-1}}$ also have critical points with no clear geometric meaning for $\phi$. It would be interesting to understand whether it is true that only the critical points that correspond to translated points of contactomorphisms (with integer time--shift) can make the norm jump, or otherwise, if this is not the case, to understand what is the correct geometric interpretation of this norm. 
\end{rmk}


\section{The discriminant metric}\label{section: discriminant metric}

In the previous section we have seen the construction of an integer--valued bi--invariant metric on the contactomorphism group of $\mathbb{R}^{2n} \times S^1$. We have also seen that, for a contact isotopy $\{\phi_t\}$ of $\mathbb{R}^{2n} \times S^1$, if the metric jumps at a time $t_0$ this is due to the presence of some discriminant point for $\phi_{t_0}$. We will now discuss, following \cite{CS}, how to generalize this idea to define an integer--valued bi--invariant metric on the universal cover of the contactomorphism group of any contact manifold.

Let $\{\phi_t\}_{t\in I}$ (for some time interval $I$) be a contact isotopy of a contact manifold $(M,\xi)$. For simplicity we assume that $M$ is compact (in the case of compactly supported contactomorphisms of a non--compact contact manifold all arguments and definitions in this section are analogous, by considering only the interior of the support). We will say that $\{\phi_t\}_{t\in I}$ is an \emph{embedded contact isotopy} if, for all $t$ and $t'$ in $I$ (with $t \neq t'$), the composition $\phi_t \circ \phi_{t'}^{-1}$ has no discriminant points. In other words, $\{\phi_t\}_{t\in I}$ is an embedded contact isotopy if and only if the submanifold $\bigcup_{t\in I} \text{gr}(\phi_t)$ of $M \times M \times \mathbb{R}$ is embedded. The definition of the discriminant metric is based on the following key lemma.

\begin{lemma}\label{lemma: decomposition}
Let $\{\phi_t\}_{t\in [0,1]}$ be a contact isotopy of a contact manifold $(M,\xi)$. After perturbing $\{\phi_t\}_{t\in [0,1]}$ in the same homotopy class with fixed endpoints, there is a positive integer $N$ and a subdivision $0 = t_0 < t_1 < \cdots < t_{N-1} < t_N = 1$ such that, for all $i = 0$, $\cdots$, $N-1$, the contact isotopy $\{\phi_t\}_{t\in [t_i,t_{i+1}]}$ is embedded.
\end{lemma}

Note that in the symplectic case the analogue of the above lemma does not work. Indeed, because of the Arnold conjecture we cannot get rid of fixed points of Hamiltonian isotopies, not even on a small time interval. As we discussed in Section \ref{section: Rigidity and flexibility in contact topology}, in the contact case we do have an analogue of the Arnold conjecture but the statement involves translated points of contactomorphisms. On the other hand, as we have seen, discriminant points of contactomorphisms generically do not exist, and this is what gives us the room to prove Lemma \ref{lemma: decomposition} (see \cite{CS} for details).  

We define the \textbf{discriminant norm} 
$$
\lVert [\{\phi_t\}] \rVert_{\text{discr}}
$$ 
of an element in the universal cover of the contactomorphism group  of $(M,\xi)$ which is different than the identity as the minimal number of pieces in a representation as in Lemma \ref{lemma: decomposition}. We also declare the discriminant norm of the identity to be zero\footnote{In other words, the discriminant metric is the word metric on $\widetilde{\text{Cont}}_0(M,\xi)$ with respect to the generating set formed by embedded (non--constant) contact isotopies (see Gal end Kedra \cite{GK}).}. It is immediate to prove that this definition gives rise to a bi--invariant metric on $\widetilde{\text{Cont}}_0(M,\xi)$. Note in particular that bi--invariance follows from the fact that, as discussed in Section \ref{section: Rigidity and flexibility in contact topology}, discriminant points are invariant by conjugation. Note also that the definition of this metric does not depend on the choice of a contact form for $\xi$, since the notion of discriminant point does not depend on this choice. Again we refer to \cite{CS} for more details.

The definition of $\lVert \, \cdot \, \rVert_{\text{discr}}$ and the proof of the fact that it is a conjugation--invariant norm only rely on elementary arguments. On the other hand the deep result about the discriminant metric is that, at least for certain contact manifolds, it is not equal (nor equivalent) to the trivial metric. Lemma \ref{lemma: decomposition} says that \emph{locally} we can always get rid of discriminant points. If it was true that by deforming any given contact isotopy in the same homotopy class we could always get rid of all discriminant points in the whole time interval, then this would mean that the discriminant metric is always equal to the trivial metric. On the other hand if, in the universal cover of some contact manifold, we find an element $[\{\phi_t\}]$ for which the discriminant length is bigger than one then this means that, no matter how we deform $\{\phi_t\}$ in its homotopy class, we are always obliged to have discriminant points at some time. This is then a global rigidity result that contrasts with the local flexibility given by Lemma \ref{lemma: decomposition}. In \cite{CS} we show that such examples exist, in particular we prove that for $\mathbb{R}^{2n} \times S^1$ and $\mathbb{R}P^{2n-1}$ the discriminant metric is unbounded (and thus not even equivalent to the trivial metric).

\begin{thm}[\cite{CS}]\label{theorem: unboundedness of discriminant metric}
The discriminant metric on the universal cover of the contactomorphism group of $\mathbb{R}^{2n} \times S^1$ and $\mathbb{R}P^{2n-1}$ is unbounded.
\end{thm}

Unboundedness of the discriminant metric on the universal cover of the contactomorphism group of $\mathbb{R}^{2n} \times S^1$ is proved using the spectral invariant $c^+$ that we described in Section \ref{section: spectral metric}. In fact in this case we can prove that even the norm induced on the contactomorphism group itself, i.e. 
$$
\lVert \phi \rVert_{\text{discr}} := \inf \{\, \lVert [\{\phi_t\}] \rVert_{\text{discr}} \; \text{ , } \; \phi_1 = \phi \,\}
$$
is unbounded. Indeed, suppose first that $\phi$ is a contactomorphism with $c^+(\phi) > 1$. Then for all contact isotopies $\{\phi_t\}$ with $\phi_1 = \phi$ we must have a time $t_0$ such that $c^+(\phi_{t_0}) = 1$. Since $c^+(\phi_{t_0}) = 1$, $\phi_{t_0}$ must have a translated point of time--shift $1$, i.e. a discriminant point. This proves that if $c^+(\phi) > 1$ then every contact isotopy from the identity to $\phi$ must have at least $2$ embedded pieces, and thus $\lVert \phi \rVert_{\text{discr}} \geq 2$. Suppose now that $c^+(\phi) > 2$ and let $\{\phi_t\}$ be any contact isotopy connecting $\phi = \phi_1$ to the identity. As above we know that there is a time $t_0$ such that $c^+(\phi_{t_0}) = 1$. By the triangle inequality (property (2) of Proposition \ref{proposition: properties of spectral invariants for contactomorphisms}) we also know that $c^+(\phi \circ \phi_{t_0}^{-1}) > 1$. By the same argument as above it then follows that $\{\phi_t\}_{t \in (t_0,1]}$ has at least two embedded pieces, and so $\lVert \phi \rVert_{\text{discr}} \geq 3$. By iterating this argument we see thus that if $c^+(\phi) > k$ then $\lVert \phi \rVert_{\text{discr}} \geq k+1$. Since, as we have seen in Section \ref{section: spectral metric}, there are contactomorphisms of $\mathbb{R}^{2n} \times S^1$ with arbitrarily large $c^+$ we conclude that the discriminant norm in $\mathbb{R}^{2n} \times S^1$ is unbounded.

Unboundedness of the discriminant norm on the universal cover of the contactomorphism group of $\mathbb{R}P^{2n-1}$ is proved using Givental's non--linear Maslov index \cite{Givental - Nonlinear Maslov index}. In this case we only prove that the norm is unbounded on the universal cover of the contactomorphism group, not on the contactomorphism group itself. More precisely we show that if $\{\phi_t\}$ is the $4k$-th iteration of the Reeb flow then the discriminant norm of the corresponding element of the universal cover is at least $k + 1$. Note that although the contact isotopy $\{\phi_t\}$ itself is the concatenation of $4k$ embedded pieces, in principle it might be possible to find another contact isotopy in the same homotopy class with smaller discriminant length. This is for instance what happens in the case of the sphere, where it is known that every element of the universal cover of the contactomorphism group has discriminant length at most $4$ (see Theorem \ref{theorem: boundedness} below). For projective space we prove that every contact isotopy in the same homotopy class as the $4k$-th iteration of the Reeb flow has length at least $k + 1$. The fact that in our proof we loose some of the length is related to the fact that the non--linear Maslov index is a \emph{quasimorphism} on the universal cover of the contactomorphism group and not a homomorphism (which, as it is known, cannot exist). However we do not know whether our result is sharp, i.e. whether indeed it does exist a contact isotopy of length $k+1$ which is homotopic to the $4k$-th iteration of the Reeb flow.

Theorem \ref{theorem: unboundedness of discriminant metric} should be compared with the following soft result.

\begin{thm}[\cite{CS}]\label{theorem: boundedness}
The discriminant norm is bounded for $\mathbb{R}^{2n+1}$ and $S^{2n-1}$. More precisely, every contact isotopy of $\mathbb{R}^{2n+1}$ has length at most $2$, and every contact isotopy of $S^{2n-1}$ has length at most $4$.
\end{thm}

Note that the difference in the behaviour of the discriminant metric in the cases of $\mathbb{R}^{2n+1}$ and $\mathbb{R}^{2n} \times S^1$ is similar to what happens for the non--squeezing phenomenon. For $\mathbb{R}^{2n+1}$ we have a complete flexibility, while if we make the Reeb flow 1--periodic and consider instead $\mathbb{R}^{2n} \times S^1$ then some global rigidity appears. On the other hand, our result for the sphere is consistent with the fact, proved by Fraser, Polterovich and Rosen \cite{FPR}, that on the contactomorphism group of $S^{2n-1}$ there does not exist any unbounded bi--invariant metric. We do not know whether the bounds in Theorem \ref{theorem: boundedness} are sharp, i.e. for example if there exists an element on the universal cover of the contactomorphism group of $S^{2n-1}$ of length exactly $4$.

It is still not well understood what it is that makes the discriminant metric bounded or unbounded. The case of $S^{2n-1}$ shows that having a 1--periodic Reeb flow is not enough. On the other hand an important difference between the case of $S^{2n-1}$ on  one side and $\mathbb{R}P^{2n-1}$ and $\mathbb{R}^{2n} \times S^1$ on the other side is that in the case of the sphere all closed Reeb orbits are contractible, while in the other two cases the Reeb flow generates the fundamental group of the manifold. Thus the way the Reeb flow interacts with the topology of the underlying manifold seems to play an important role in the question of whether the discriminant metric on a given contact manifold is bounded or unbounded. At the moment $\mathbb{R}^{2n} \times S^1$ and $\mathbb{R}P^{2n-1}$ are the only known cases of contact manifolds with unbounded discriminant metric. As we have discussed, both results are obtained using generating functions. Still using generating functions, and based on techniques developed by Givental \cite{Givental - Nonlinear Maslov index, Givental - toric}, in a work in progress with Gustavo Granja, Yael Karshon and Milena Pabiniak we are exploring the possibility to study the discriminant metric and other contact rigidity phenomena for more general contact toric manifolds. In particular, we are addressing the following question. Given a prequantizable symplectic toric manifold, for which prequantizations is the discriminant metric unbounded? (Note for example that both $S^{2n-1}$ and $\mathbb{R}P^{2n-1}$ are prequantizations of $\mathbb{C}P^{n-1}$, and in one case the discriminant metric is bounded while in the other it is unbounded). The reason why there are so few results about unboundedness of the discriminant metric is that at the moment it is still not known how to use $J$--holomorphic curves to study this problem. I hope that the Floer homology theory for translated points that I am developing \cite{Floer homology for translated points}, as well as the work in progress by Z\'{e}na\"{i}di \cite{Zenaidi}(which is based on Legendrian contact homology), will provide good tools for this problem.


\section{Orderability and the oscillation metric}\label{section: orderability and the oscillation metric}

In this section we will first discuss the notion of orderability, which was introduced by Eliashberg and Polterovich in \cite{EP00}, and then explain (still following \cite{CS}) how to combine this notion with the bi--invariant metric described in the previous section in order to obtain what we call the oscillation metric. In this whole section we assume that $(M,\xi)$ is compact.

Note first that on any contact manifold $(M,\xi)$ there is a natural notion of a \textbf{positive contact isotopy}. Namely we say that a contact isotopy $\{\phi_t\}$ is positive if it moves every point in a direction positively transverse to the contact distribution $\xi$. Note that, for any choice of a contact form $\alpha$ for $\xi$, the contact isotopy $\{\phi_t\}$ is positive if and only if $\alpha(X_t) > 0$ where $X_t$ is the vector field generating $\phi_t$. Thus we see that a contact isotopy is positive if and only if it is generated by a positive Hamiltonian  $H_t$ (recall that $H_t = \alpha(X_t)$). Note that this notion is specific of the contact case. Indeed in the symplectic case Hamiltonian functions are only defined up to the addition of a constant and thus, by choosing the constant big enough, any Hamiltonian isotopy of a compact symplectic manifold can be generated by a positive Hamiltonian \footnote{On the other hand, if we normalize the Hamiltonians by $\int H \, \omega^n = 0$ then they are never positive everywhere. If the symplectic manifold $W$ is open then we can normalize the Hamiltonians by requiring them to be compactly supported. In this case it makes sense to call a Hamiltonian isotopy positive if it is generated by  a Hamiltonian function that is positive in the interior of the support. Similarly it also makes sense to define in the same way positive contact isotopies for open contact manifolds. However this positivity notion for open contact manifolds misses the specifically contact flavour of positive contact isotopies on closed contact manifolds. For instance the fact, proved by Bhupal, that the above notion of positive contact isotopy induces a partial order on the group of compactly supported contactomorphisms of $\mathbb{R}^{2n+1}$ is a generalization of a fact that is already true, and proved by Viterbo, for the group of compactly supported Hamiltonian symplectomorphisms of $\mathbb{R}^{2n}$.}. On the other hand, in the contact case the most prominent example of a positive contact isotopy is given by the Reeb flow. Similarly, we say that a contact isotopy is non--negative if it moves every point in a direction positively transverse or tangent to the contact distribution, thus if and only if it is generated by a non--negative contact Hamiltonian. Using this notion we then define a relation $\leq$ on the universal cover $\widetilde{\text{Cont}}_0(M,\xi)$ of the contactomorphism group by saying that
$$
[\{\phi_t\}] \leq [\{\psi_t\}]
$$
if $[\{\phi_t\}] \cdot [\{\psi_t\}]^{-1}$ can be generated by a non--negative contact isotopy. It is easy to prove that this relation is always reflexive and transitive. However it is not always a partial order because anti--symmetry can fail, i.e. it is not always true that if $[\{\phi_t\}] \leq [\{\psi_t\}]$ and $[\{\psi_t\}] \leq [\{\phi_t\}]$ then $[\{\phi_t\}] = [\{\psi_t\}]$. If the relation $\leq$ is a partial order on $\widetilde{\text{Cont}}_0(M,\xi)$ then we say that the contact manifold $(M,\xi)$ is \textbf{orderable}. It follows immediately from the definition that $(M,\xi)$ is orderable if and only if there are no non--negative non--constant contractible loops of contactomorphisms. Eliashberg and Polterovich \cite{EP00} proved on the other hand that if for a given contact manifold there is a non--constant non--negative contractible loop of contactomorphisms then this loop can be deformed to a positive contractible loop. Thus we see that a contact manifold is orderable if and only if it does not have any positive contractible loop of contactomorphisms \cite[Criterion 1.2.C]{EP00}. The first contact manifold that was proved to be orderable is $\mathbb{R}P^{2n-1}$. As observed by Eliashberg and Polterovich \cite{EP00}, in this case the non--existence of a positive contractible loop of contactomorphisms follows from the properties of Givental's non--linear Maslov index. On the other hand Eliashberg, Kim and Polterovich \cite{EKP} proved that the sphere $S^{2n-1}$ is non--orderable, by constructing explicitely a positive contractible loop of contactomorphisms (see also Giroux \cite{Giroux} for a different description of this loop). By now many more examples are known of orderable and non--orderable contact manifolds, but it is still unclear what it is that makes a contact manifold orderable or not.

\begin{rmk}\label{remark: orderability and squeezing}
As discussed by Eliashberg, Kim and Polterovich \cite{EKP}, orderability is related to contact (non--)squeezing. The link between these two phenomena is given by the fact that if $(W,\omega)$ is an exact symplectic manifold and $M$ an hypersurface of contact type in it, then orderability of $M$ is related to the impossibility to squeeze by contact isotopies certain domains in the prequantization $W \times S^1$. More precisely, if $M$ is non--orderable then a positive contractible loop of contactomorphisms of $M$ can be used as a tool to squeeze \emph{small} domains in $W \times S^1$. For instance it is through this construction that Eliashberg, Kim and Polterovich proved, using the positive contractible loop of contactomorphisms they had found for $S^{2n-1}$, the squeezing result for domains of $\mathbb{R}^{2n} \times S^1$ of the form $B^{2n}(R) \times S^1$ for $R < 1$.
\end{rmk}

The notion of orderability is very much related to bi--invariant metrics on the contactomorphism group. A first link is given by the fact, observed by Fraser, Polterovich and Rosen \cite{FPR}, that if $\big(M,\xi = \text{ker}(\alpha)\big)$ is orderable and the Reeb flow is 1--periodic then the partial order $\leq$ gives rise to a bi--invariant metric. Indeed for an element $[\{\phi_t\}]$ of $\widetilde{\text{Cont}}_0(M,\xi)$ we can define $\nu_{FPR}^+\big([\{\phi_t\}]\big)$ as the minimal integer $k$ such that $e^k \geq [\{\phi_t\}]$, where $e = [\{e_t\}_{t\in [0,1]}]$ denotes the class represented by the Reeb flow , and $\nu_{FPR}^-\big([\{\phi_t\}]\big)$ as the maximal integer $k$ such that $e^k \leq [\{\phi_t\}]$ . Then it is easy to prove that 
$$
\lVert [\{\phi_t\}] \rVert_{FPR} := \max \Big(\lvert\nu_{FPR}^-\big([\{\phi_t\}]\big)\rvert, \lvert \nu_{FPR}^-\big([\{\phi_t\}]\big)\rvert \Big)
$$
is a conjugation--invariant norm. As observed by Borman and Zapolsky \cite{BZ}, this bi--invariant metric is always (stably) unbounded (assuming that $M$ is compact), because for any integer $k$ we have that $\lVert e^k \rVert_{FPR} = \lvert k \rvert$. Fraser, Polterovich and Rosen considered also the case of open contact manifolds and found some condition that guarantees unboundedness of the metric, in terms of intersection properties of associated sets in the symplectization of $M$.

We will now discuss how to combine the notion of orderability with the bi--invariant metric defined in Section \ref{section: discriminant metric} to obtain a bi--invariant pseudo--metric on the universal cover of the contactomorphism group of any contact manifold, which is non--degenerate if and only if the contact manifold is orderable. As discussed in \cite{CS}, Lemma \ref{lemma: decomposition} can be improved to show that it is possible to represent any given contact isotopy $\{\phi_t\}$, after perturbing in its homotopy class, as a concatenation of a finite number of embedded contact isotopies that are either non--negative or non--positive (not necessarily with alternating signs). We now define 
$\nu^+\big([\{\phi_t\}]\big)$ to be the minimal number of non--negative pieces in such a decomposition, and $\nu^-\big([\{\phi_t\}]\big)$ to be minus the minimal number of non--positive ones (note that the representative of $[\{\phi_t\}]$ that minimize the number of non--negative pieces is not necessarily the same as the one that minimizes the number of non--positive ones).  We then define the \textbf{oscillation pseudonorm} of $[\{\phi_t\}]$ as
$$
\lVert [\{\phi_t\}] \rVert_{\text{osc}} = \nu^+\big([\{\phi_t\}]\big) - \nu^-\big([\{\phi_t\}]\big)\,.
$$
For example, suppose that $\{\phi_t\}_{t \in [0,1]}$ is a positive contractible loop of contactomorphisms on a (non--orderable) contact manifold $(M,\xi)$. Then $\nu^-\big([\{\phi_t\}_{t\in [0,\frac{1}{2}]}]\big) = 0$, because $\{\phi_t\}_{t\in [0,\frac{1}{2}]}$ is positive and so the minimal number of non--positive pieces is zero. On the other hand, since $\{\phi_t\}_{t \in [0,1]}$ is contractible we have that $\{\phi_t\}_{t\in [0,\frac{1}{2}]}$ is in the same homotopy class as the reverse of $\{\phi_t\}_{t\in [\frac{1}{2},1]}$, which is a negative contact isotopy. Thus we also have $\nu^+\big([\{\phi_t\}_{t\in [0,\frac{1}{2}]}]\big) = 0$, and so the oscillation pseudonorm of $\{\phi_t\}_{t\in [0,\frac{1}{2}]}$ is zero. This shows that if $(M,\xi)$ is non--orderable then the oscillation pseudonorm on $\widetilde{\text{Cont}}_0(M,\xi)$ is degenerate, since we have found an element which is not the identity but has norm equal to zero. By similar arguments in \cite{CS} we prove the following result.

\begin{prop}
The oscillation pseudonorm on $\widetilde{\text{Cont}}_0(M,\xi)$ is non--degenerate if and only if $(M,\xi)$ is orderable. Moreover, in this case the oscillation norm is compatible with the partial order.
\end{prop}

As proved in \cite{CS}, the oscillation (pseudo)norm is bounded for $S^{2n-1}$ and unbounded for $\mathbb{R}P^{2n-1}$. Unboundedness for $\mathbb{R}P^{2n-1}$ follows from the same proof as in the case of the discriminant metric, combined with the fact that the non--linear Maslov index is monotone with respect to the partial order of Eliashberg and Polterovich.

Note that the discriminant metric is non--degenerate on any contact manifold just by construction. On the other hand, non--degeneracy of the oscillation metric is equivalent to orderability of the contact manifold, which is a deep and hard to prove property. This is similar to what happens in the symplectic case for the Hofer metric, whose non--degeneracy is a deep result proved in \cite{Hofer} and \cite{Lalonde-McDuff} using hard methods of symplectic topology. As discussed in \cite{Lalonde-McDuff}, non--degeneracy of the Hofer metric is deeply related to the symplectic non--squeezing theorem. Something similar also happens for our metric. Indeed, non--degeneracy of the oscillation metric is equivalent to orderability, which, as mentioned in Remark \ref{remark: orderability and squeezing}, is deeply related to the contact non--squeezing phenomenon. It would be interesting to explore further this analogy, and find concrete links between the Hofer and oscillation norms.


\section{Discussion}\label{section: Discussion}

We have seen above the construction of several bi--invariant metrics on the contactomorphism group. It is still an open question to understand whether these metrics (in the cases when they can be defined on the same contact manifold) are equivalent or not. Note that the discriminant metric is certainly not equivalent to the oscillation (pseudo)metric in general, since the oscillation pseudometric is degenerate on non--orderable contact manifolds while the discriminant metric is always non--degenerate. On the other hand, in order to compare for example the oscillation metric to the metric of Fraser, Polterovich and Rosen, as a first step it would be important to understand the following question. Let $\big(M,\xi = \text{ker}(\alpha)\big)$ be an orderable compact contact manifold with 1--periodic Reeb flow, and denote by $e = [ \{e_t\}_{t \in [0,1]}]$ the class in $\widetilde{\text{Cont}}_0(M,\xi)$ represented by the Reeb flow. As mentioned above, we have that $\lVert e^k \rVert_{FPR} = \lvert k \rvert$. On the other hand, is there some proportionality between $\lVert e^k \rVert_{osc}$ and $\lvert k \rvert$? In the case when $M = \mathbb{R}P^{2n-1}$ we know that, for every $k$, $\lvert k \rvert \leq \lVert e^{4k} \rVert_{osc} \leq 4 \lvert k \rvert$ and thus in this case we have $ \frac{1}{4} \lVert e^{4k} \rVert_{FPR} \leq \lVert e^{4k} \rVert_{osc} \leq \lVert e^{4k} \rVert_{FPR}$. However the proof of this fact that is given in \cite{CS}, using the work of Givental \cite{Givental - Nonlinear Maslov index}, is very specific of the case of projective space (in particular, the topology of $\mathbb{R}P^{2n-1}$ plays a role, since a key ingredient in the proof is the fact that the cohomology ring of $\mathbb{R}P^{2n-1}$ with $\mathbb{Z}_2$--coefficients is generated by one element). Note that in general for a compact contact manifold with 1--periodic Reeb flow there does not need to be any proportionality between $\lVert e^k \rVert_{osc}$ and $\lvert k \rvert$, since for example we have seen that on the sphere we always have $\lVert e^k \rVert_{osc} \leq 4$. It seems possible that even in the orderable case there might be no proportionality in general between $\lVert e^k \rVert_{osc}$ and $\lvert k \rvert$, hence no equivalence between the oscillation norm and the norm of Fraser, Polterovich and Rosen.

As mentioned at the end of the previous section, another important question would be to understand whether there is some concrete connection between the oscillation norm and the Hofer norm, for example in the case when the contact manifold is the prequantization of a symplectic manifold. A question that might help to find such connection is the following. Let $\lVert \, \cdot \, \rVert_{\alpha}$ be the Hofer--like norm on $\widetilde{\text{Cont}}_0(M,\xi)$ that we discussed in Section \ref{section: Hofer-like lengths}, with respect to a certain contact form $\alpha$ for $\xi$ (note that any other choice of a contact form for $\xi$ gives rise to an equivalent norm, see \cite{MSpaeth, Egor}). As we have seen in Section \ref{section: Hofer-like lengths}, this norm is not conjugation--invariant. However for a class $[\{\phi_t\}]$ in $\widetilde{\text{Cont}}_0(M,\xi)$ we can try to define
$$
\lVert [\{\phi_t\}] \rVert_{\text{conj},\alpha} := \sup_{\psi} \lVert [\{\psi \phi_t \psi^{-1}\}] \rVert_{\alpha} \,.
$$
In \cite{Egor} Shelukhin raised the question of whether, at least for certain contact manifolds, the above expression might be well--defined (i.e. whether it might be true that the supremum is always a finite real number). Note that no examples nor counterexamples are known at the moment for this question. If on some contact manifold the above expression turns out to be well--defined  then it gives rise to a conjugation--invariant norm and so, by Proposition \ref{proposition: Polterovich}, it must be discrete. In this case it would be interesting to understand what happens when $\lVert \, \cdot \, \rVert_{\text{conj},\alpha}$ jumps. Is a jump of $\lVert \, \cdot \, \rVert_{\text{conj},\alpha}$ also related to existence of discriminant points? If so, does this help to find a relation between the Hofer and the oscillation norms?
 
As discussed by Polterovich \cite{Polterovich - book}, one of the ideas of why it is good to have a bi--invariant metric on the Hamiltonian group is that we want to study the dynamics of Hamiltonian flows. Instead of looking at all orbits of a given Hamiltonian flow, we can see the flow as a single geometric object, a curve in the Hamiltonian group. We can then hope that geometric properties of this curve (with respect to our bi--invariant metric) will reflect some aspects of the dynamics of the flow. See \cite{Polterovich - book} for instances of how this idea can be exploited, in particular in connection to the relation between growth and dynamics (Chapter 8) and to ergodic theory (Chapter 11). In the contact case, can we use the bi--invariant metrics described in this paper in a similar way? 

A second reason why it is useful to have bi--invariant metrics on the contactomorphism group is that, similarly to the symplectic case, they can be used to define notions of size (or \emph{contact capacity}) for subsets of a given contact manifold. Consider for instance the oscillation metric. Define the capacity $c(\mathcal{U})$ of a domain $\mathcal{U}$ in a given contact manifold to be the supremum of the values $\nu^+ \big([\{\phi_t\}]\big)$ for all contact isotopies $\{\phi_t\}$ supported in $\mathcal{U}$, and its displacement energy $E(\mathcal{U})$ to be the infimum of the values of $\lVert \{\psi_t\}\rVert_{osc}$ for all contact isotopies $\{\psi_t\}$ displacing $\mathcal{U}$, i.e. such that $\psi_1(\mathcal{U}) \cap \mathcal{U} = \emptyset$. Note that both the capacity and the displacement energy are monotone with respect to inclusion, and invariant by the action of the contactomorphism group. We expect that, in analogy to the symplectic case, the energy--inequality capacity $c(\mathcal{U}) \leq E(\mathcal{U})$ should hold for all domains (proving in particular that the capacity is always well--defined). In the symplectic case the energy--capacity inequality is proved using Gromov's non--squeezing theorem, and it implies non--degeneracy of the Hofer metric \cite{Lalonde-McDuff}. As we have discussed above, in the contact case we also have a relation (via orderability) between the contact non--squeezing phenomenon and non--degeneracy of the oscillation metric. It would be interesting to understand whether, similarly to the symplectic case, some role in this relation is also played by the energy--capacity inequality. Note that in the symplectic case Lalonde and McDuff \cite{Lalonde-McDuff, Lalonde-McDuff - Local non-squeezing} also studied, more generally, a link between length--shortening of paths in the Hamiltonian group (with respect to the Hofer metric) and symplectic non--squeezing. In the contact case, do we also have a similar link between length--shortening of paths in the contactomorphism group, for instance with respect to the oscillation metric, and the contact (non--)squeezing phenomenon? Note that the link found by Eliashberg, Kim and Polterovich \cite{EKP} between orderability and (non--)squeezing might be interpreted as a first step in this direction. Indeed a positive contractible loop of contactomorphisms can be seen as a curve in the contactomorphism group that can be shortened with respect to the oscillation length: the loop itself has positive length, but it can be shortened in its homotopy class to the constant loop, which has zero length. As discussed by Eliashberg, Kim and Polterovich, a homotopy between a positive contractible loop of contactomorphisms and the constant loop induces a contact squeezing between some associated domains. It would be interesting to understand if this result is indeed the special case of a more general link between length--shortening with respect to the oscillation metric and contact (non--)squezing, and to explore analogies and differences with respect to the corresponding picture in the symplectic case.

In the symplectic case, geodesics with respect to the Hofer norm have been studied by many authors, see in particular \cite{Lalonde-McDuff - Hofer's geometry, BP - Geodesics, Siburg, Long, Ustilovsky, MS - HZ capacity and, KL, KS - Shortening}. In the contact case the question, with respect to any of the metrics described in this article, is still completely unexplored. Other problems that have been studied for the Hofer norm but are still open in the contact case are, for instance, the problem of studying the length spectrum of $(M,\xi)$, i.e. the set of values $\lVert [\{\phi_t\}] \rVert$ for elements $ [\{\phi_t\}]$ of $\pi_1\big(\text{Cont}(M,\xi)\big)$ (cf. in the symplectic case \cite[Chapter 9]{Polterovich - book}) and the problem of studying all contactomorphisms of a given contact manifold $(M,\xi)$ seen as isometries of $\text{Cont}_0(M,\xi)$ (cf. in the symplectic case \cite[Chapter 14]{Polterovich - book} and \cite{Lalonde - Polterovich}).

In \cite{CS} we also define the discriminant and oscillation lengths for Legendrian isotopies. Indeed we show that, after deforming it in the same homotopy class, every Legendrian isotopy can be written as the concatenation of a finite number of embedded monotone pieces, and define the discriminant and oscillation lengths as in the case of contactomorphisms. Let $\mathcal{L}(M \times M \times \mathbb{R}, \Delta)$ be the space of Legendrian submanifolds of the contact product $M \times M \times \mathbb{R}$ which are Legendrian isotopic ot the diagonal $\Delta$, and consider the embedding
$$
\text{Cont} (M,\xi) \rightarrow \mathcal{L}(M \times M \times \mathbb{R}, \Delta)
$$
that sends a contactomorphism to its contact graph. This embedding preserves the (discriminant or oscillation) length of paths in the two spaces. However the induced map
$$
\widetilde{\text{Cont}}_0 (M,\xi) \rightarrow \widetilde{\mathcal{L}}(M \times M \times \mathbb{R}, \Delta)
$$
does not necessarily preserve the discriminant or oscillation lengths, because the infimum on the right hand side is taken on a larger set (not all Legendrian isotopies are graphs). In analogy to what was proved by Ostrover \cite{Ostrover - A comparison} for the Hofer norm, we expect the above map to be far from being an isometric embedding. In general the discriminant and oscillation lengths for Legendrian isotopies seem to behave quite differently from the case of contactomorphisms \footnote{I would like to thank Emmy Murphy for first opening my eyes on this.}. For instance, although the same arguments as in the case of contactomorphisms show that the Legendrian discriminant and oscillation lengths are unbounded on $\mathbb{R}^{2n} \times S^1$ and $\mathbb{R}P^{2n-1}$, in the Legendrian case we do not have any example where we can prove that the lengths are bounded. Even for the standard Euclidean space $\mathbb{R}^{2n+1}$ it is not clear to us that this should be the case.

Finally, in \cite{CN} Chernov and Nemirovski discussed a link between orderability and the notion of causality in Lorentz geometry. It would be interesting to explore whether the discriminant and oscillation lengths for contact and Legendrian isotopies, and the known results about them, also have some relevant interpretations in this context.


\end{document}